\let\ORIlabel\label
\let\ORIrefstepcounter\refstepcounter
\AddToHook{package/hyperref/before}{%
  \let\label\ORIlabel
  \let\refstepcounter\ORIrefstepcounter
}

\documentclass[onefignum,onetabnum]{siamonline171218}

\usepackage{mathpazo}
\usepackage{mathtools}
\usepackage{lipsum}
\usepackage{amsfonts}
\usepackage{graphicx}
\usepackage{epstopdf}
\usepackage{algorithmic}
\usepackage{amssymb}
\ifpdf
  \DeclareGraphicsExtensions{.eps,.pdf,.png,.jpg}
\else
  \DeclareGraphicsExtensions{.eps}
\fi
\newcommand{\dd}{\mathrm{d}}
\newcommand{\ee}{\mathrm{e}}
\newcommand{\ii}{\mathrm{i}}
\renewcommand{\varepsilon}{\epsilon}
\usepackage{enumitem}
\setlist[enumerate]{leftmargin=.5in}
\setlist[itemize]{leftmargin=.5in}

\newsiamremark{remark}{Remark}
\newsiamremark{hypothesis}{Hypothesis}
\crefname{hypothesis}{Hypothesis}{Hypotheses}
\newsiamthm{claim}{Claim}
\newsiamthm{example}{Example}
\newsiamthm{siamprop}{Proposition}

\headers{Linearized Self-Similar Universality for Benjamin-Ono}{L. Gassot, P. G\'erard, and P. D. Miller}

\title{The Benjamin-Ono Equation in the Long-Time Limit:  Linearized Self-Similar Universality
}

\author{Louise Gassot\thanks{CNRS and Department of Mathematics, University of Rennes, Rennes, France
	(\email{louise.gassot@cnrs.fr}).}
 \and
 Patrick G\'erard\thanks{Department of Mathematics, Universit\'e Paris Saclay, Paris, France (\email{patrick.gerard@universite-paris-saclay.fr}).}
\and Peter D. Miller\thanks{Department of Mathematics, University of Michigan, Ann Arbor, MI 
  (\email{millerpd@umich.edu}).}
}

\usepackage{subcaption}

\usepackage{amsopn}

\makeatletter
\newcommand*{\addFileDependency}[1]{
  \typeout{(#1)}
  \@addtofilelist{#1}
  \IfFileExists{#1}{}{\typeout{No file #1.}}
}
\makeatother

\ifpdf
\hypersetup{
  pdftitle={The Benjamin-Ono Equation in the Long-Time Limit:  Linearized Self-Similar Universality},
  pdfauthor={Louise Gassot and Patrick Gérard and Peter D. Miller}
}
\fi

\usepackage{tikz}
\usetikzlibrary{arrows.meta}
\usetikzlibrary{decorations.pathmorphing}
\usetikzlibrary{decorations.markings}
\usetikzlibrary{patterns}
\usetikzlibrary{automata}
\usetikzlibrary{positioning}
\usepackage{tikz-cd}

\tikzset{->-/.style={decoration={markings,mark=at position #1 with {\arrow[thick]{>}}},postaction={decorate}}}
	
\tikzset{-<-/.style={decoration={markings,mark=at position #1 with{\arrow[thick]{<}}},postaction={decorate}}}

\title{The Benjamin-Ono Equation in the Long-Time Limit:  Linearized Self-Similar Universality}

\begin{document}

\maketitle

\begin{abstract}
We obtain the leading term in the solution of the Cauchy problem for the Benjamin-Ono equation in the limit $t\to+\infty$ with $x=O(t^{1/2})$.  We show that the rate of decay exceeds that of self-similar solutions and obtain an explicit universal profile for the decaying solution, relating it to the linearization of the profile equation for self-similar solutions.  The proof assumes a class of rational initial data $u_0$ in $L^2(\mathbb{R})\cap L^1(\mathbb{R})$ that exhibit generic behavior of the reflection coefficient at the origin.  
\end{abstract}



\section{Introduction}
It is by now well known \cite{AblowitzS77,DeiftVZ94} that for the Korteweg-de Vries (KdV) equation in the form 
\begin{equation}
    \partial_t u-3\partial_x u^2+\partial_x^3u=0
    \label{eq:KdV}
\end{equation}
with Schwartz-class initial data $u_0$, the long time behavior of the solution consists mainly of a dispersive region where $x/t<-\delta$ (the Zakharov-Manakov or similarity region) and a soliton region where $x/t>\delta$ for some $\delta>0$.

According to \cite{AblowitzS77,DeiftVZ94}, the behavior of the KdV solution for large $t$ when $x/t$ is small is rather complicated, and part of the story involves self-similar solutions of \eqref{eq:KdV}.  As was first worked out carefully in \cite{SegurA81}, these are based on the ansatz $u(t,x)=bt^pU(\xi)$, with a similarity variable $\xi=x/(at^q)$.  Substitution into \eqref{eq:KdV} results in 
\begin{equation}
bpt^{p-1}U(\xi)-bq t^{p-1}\xi U'(\xi)-6\frac{b^2}{a}t^{2p-q}U(\xi)U'(\xi)+\frac{b}{a^3}t^{p-3q}U'''(\xi)=0.
\end{equation}
Balancing the powers of $t$ gives $p=-\frac{2}{3}$ and $q=\frac{1}{3}$.  Furthermore choosing $a=3^{1/3}$ and $b=3^{-2/3}$ for convenience yields the profile equation
\begin{equation}
\xi U'(\xi) + 2U(\xi)+6U(\xi)U'(\xi)-U'''(\xi)=0.
\label{eq:KdV-profile}
\end{equation}
To solve this equation, one can introduce the Miura transformation $U(\xi)=V(\xi)^2 + V'(\xi)$, leading to
\begin{equation}
\left(2V(\xi) + \frac{\dd}{\dd\xi}\right)\frac{\dd}{\dd\xi}\left(V''(\xi)-2V(\xi)^3-\xi V(\xi)\right)=0.
\end{equation}
Thus, whenever $V(\xi)$ is a solution of the homogeneous Painlev\'e-II equation in the form
\begin{equation}
    V''(\xi)-2V(\xi)^3-\xi V(\xi)=0,
\label{eq:PII}
\end{equation}
the profile $U(\xi)=V(\xi)^2+V'(\xi)$ satisfies \eqref{eq:KdV-profile} and hence determines a self-similar solution $u(t,x)=(3t)^{-2/3}U(x/(3t)^{1/3})$ of \eqref{eq:KdV}.

In the asymptotic regime $t\to+\infty$ with $\xi=x/(3t)^{1/3}$ bounded (so $x/t=O(t^{-2/3})$), it is known that the solution $u(t,x)$ of \eqref{eq:KdV} is accurately approximated by a self-similar solution based on different solutions of the Painlev\'e-II equation \eqref{eq:PII} depending on the value at $\lambda=0$ of the reflection coefficient $r(\lambda)$ from the direct scattering problem (stationary Schr\"odinger equation) with potential $u_0$.  In the generic case one has $r(0)=-1$, and the relevant solution of \eqref{eq:PII} is universally given by the \emph{Hastings-McLeod} solution $V=V_\mathrm{HM}(\xi)$.  This solution is distinguished by its asymptotic behavior as $\xi\to+\infty$, wherein $V_\mathrm{HM}(\xi)\sim -\mathrm{Ai}(\xi)$.  In the opposite limit one has instead that $V_\mathrm{HM}(\xi)$ grows proportional to $\sqrt{-\xi}$ as $\xi\to-\infty$.  In the non-generic case one has instead a strict inequality $|r(0)|<1$, and the relevant solution of \eqref{eq:PII} is one of the \emph{Ablowitz-Segur} solutions $V=V_\mathrm{AS}(\xi)$ determined by the asymptotic $V_\mathrm{AS}(\xi)\sim r(0)\mathrm{Ai}(\xi)$ as $\xi\to+\infty$.  By contrast with the Hastings-McLeod solution, $V_\mathrm{AS}(\xi)$ exhibits oscillatory decay as $\xi\to -\infty$ qualitatively similar to that of a multiple of $\mathrm{Ai}(\xi)$.

In the generic case $r(0)=-1$, there is also a somewhat larger region called the \emph{collisionless shock region}, wherein $u(t,x)$ is described by a modulated nonlinear wavetrain of diminishing amplitude proportional to $(\ln(t)/t)^{2/3}$.  This region can be found between the dispersive region $x<-\delta t$ and the self-similar region $|x|=O(t^{-1/3})$, and it occupies scales where $x$ is negative and proportional to $t^{1/3}\ln(t)^{2/3}$.  In the non-generic case $|r(0)|<1$ there is no need for the collisionless shock region and the self-similar and dispersive approximations formally match.  The known asymptotic results omit certain thin transitional layers between the four regions, and there is particular interest in resolving the solution of \eqref{eq:KdV} between the collisionless shock region and the self-similar region in the generic case.

We now turn to the large-time limit for the Benjamin-Ono (BO) equation instead:
\begin{equation}
\partial_t u + \partial_x(u^2)=\partial_x|D|u,\quad \widehat{|D|u}(\diamond)=|\!\diamond\!|\cdot\widehat{u}(\diamond).  
\label{eq:XS-BO}
\end{equation}
The asymptotic behavior of solutions of this equation in regions corresponding to the dispersive and soliton regions has only recently been settled \cite{GassotGM26}.  The behavior of the solution in these two regions, again given by $x<-\delta t$ and $x>\delta t$ for some arbitrarily small fixed $\delta>0$, is sufficient to globally characterize the solution as $t\to+\infty$ in the $L^2(\mathbb{R})$ sense.  However, the question of universality of asymptotic behavior in the regime $t\to+\infty$ with $x/t$ small has not been addressed.  Based on the KdV problem and other integrable equations, one might expect self-similar solutions to play an important role.  To characterize them for BO, we substitute the ansatz $u(t,x)=bt^pU(\xi)$ with $\xi:=x/(at^q)$ into \eqref{eq:XS-BO} and accounting for the scale invariance of the Hilbert transform one obtains 
\begin{equation}
bpt^{p-1}U(\xi)-bqt^{p-1}\xi U'(\xi) +\frac{2b^2}{a}t^{2p-q}U(\xi)U'(\xi) =\frac{b}{a^2}t^{p-2q}\frac{\dd}{\dd\xi}|D|U(\xi),
\end{equation}
where $|D|$ now denotes a Fourier multiplier with respect to the $\xi$-variable.
Matching the powers of $t$ now requires $p=-\frac{1}{2}$ and $q=\frac{1}{2}$.  Then taking $a=2$ and $b=1$ for convenience, one arrives at the profile equation
\begin{equation}
    -2U(\xi)-2\xi U'(\xi) + 4U(\xi)U'(\xi)=\frac{\dd}{\dd \xi}|D|U(\xi).
\label{eq:BO-profile}
\end{equation}
For profiles $U(\xi)$ that tend to zero sufficiently rapidly in one or the other limit $\xi\to\pm\infty$ we may integrate with zero integration constant to obtain
\begin{equation}
    -2\xi U(\xi) +2U(\xi)^2 = |D|U(\xi).
    \label{eq:integrated-BO-profile}
\end{equation}
Therefore, one might expect that in the limit $t\to+\infty$ with $x=O(t^{1/2})$, the solution $u(t,x)$ of \eqref{eq:XS-BO} for a large class of initial data would be accurately approximated by a self-similar solution $u=t^{-1/2}U(x/(2t^{1/2}))$ with $U(\xi)$ satisfying \eqref{eq:integrated-BO-profile}.  

As in the case of the Korteweg-de Vries equation, there is a dichotomy in the inverse-spectral theory of the Benjamin-Ono equation:  a given initial datum $u_0$ in a weighted $L^2$ space is either generic or non-generic.  The dichotomy again involves the reflection coefficient $\beta(\lambda)$ defined for  $\lambda>0$ now from a different (nonlocal) Lax equation. In particular the asymptotic behavior of $\beta(\lambda)$ as $\lambda\downarrow 0$ is of crucial importance.  The two alternatives are:
\begin{itemize}
    \item $u_0$ is generic if and only if 
    \begin{equation}
        \beta(\lambda)=\frac{2\pi\ii}{\ln(\lambda)}(1+o(1)),\quad\lambda\downarrow 0.
    \label{eq:XS-beta-generic}
    \end{equation}
    \item $u_0$ is nongeneric if and only if 
    \begin{equation}
        \beta(\lambda)=o\left(\frac{1}{\ln(\lambda)}\right),\quad\lambda\downarrow 0.
    \end{equation}
    The rate of decay is given as $O(\lambda)$ in \cite{KaupM98} and $O(\lambda^\epsilon\ln(\lambda))$ for some $\epsilon\in (0,1)$ in \cite{Wu17}.
\end{itemize}
In particular, all multisoliton solutions are nongeneric since they are reflectionless potentials for which $\beta(\lambda)$ vanishes identically for $\lambda>0$.

\subsection{Results}
Fix distinct points $p_1,\dots,p_N$ in the open upper half-plane and complex constants $c_1,\dots,c_N$ with the property that 
\begin{equation}
\mathrm{Re}(c_1+\cdots+c_N)=0.      
\label{eq:XS-u0-in-L1}
\end{equation}
In this paper we consider the Benjamin-Ono equation \eqref{eq:XS-BO} with rational initial data $u_0(x)=u(0,x)$ of the form
\begin{equation}
        u_0(x)=\sum_{n=1}^N\left[\frac{c_n}{x-p_n} + \frac{c_n^*}{x-p_n^*}\right],\quad x\in\mathbb{R}.
    \label{eq:XS-u0-rational}
\end{equation}
The condition \eqref{eq:XS-u0-in-L1} guarantees that $u_0\in L^1(\mathbb{R})$.  Equivalently, the sum of the residues at all of the poles vanishes, so as a function of the complex variable $z$, $u_0$ has a single-valued antiderivative $L(z)$ defined in the neighborhood of $z=\infty$ normalized so that $L(\infty)=0$.

We define a quantity $\Delta$ that depends on the pole and residue data as follows.  First, we say that an index $n=1,\dots,N$ is \emph{exceptional} if $\ii c_n$ is a negative integer:  $\ii c_n=-1,-2,-3,\dots$.  Otherwise $n$ is \emph{non-exceptional}.  We will assume without loss of generality that if there exists at least one non-exceptional index, then the poles are ordered so that $N$ is non-exceptional.  If $n$ is a non-exceptional index, we let $C_n$ denote a closed oriented Jordan curve on the Riemann sphere punctured at $p_1,\dots,p_N,p_1^*,\dots,p_N^*$ passing through $z=\infty$ and $z=0$ (interior/exterior on the left/right) such that $p_1,\dots,p_n$ lie in the interior while all other poles of $u_0$ are in the exterior.  Then we write $C_n=C_{n}^{-}\sqcup C_{n}^{+}\sqcup\{0,\infty\}$ where $C_{n}^{-}$ denotes the arc of $C_n$ originating at $z=\infty$ and terminating at $z=0$.  By the residue theorem, analytic continuation of $L(z)$ along $C_{n}$ from $z=\infty$ with initial value $L(\infty-)=0$ back to $z=\infty$ results in a limiting value of $L(\infty+)=2\pi\ii (c_1+\cdots + c_n)$. Setting 
\begin{equation}
  E_n:=\ee^{-\ii L(\infty+)} =\ee^{2\pi(c_1+\cdots+c_n)},\quad n=1,\dots,N,\quad \text{for $n$ non-exceptional}, 
\label{eq:XS-En-non-exceptional}
\end{equation}
we see that $\ee^{-\ii L(z)}$ varies from $1$ to $E_n$ as $z$ traverses $C_n$ from $z=\infty$ and back again.  If instead $n$ is exceptional, we let $C_n$ be an oriented Jordan arc on the punctured Riemann sphere from $z=\infty$ to $z=p_n$ passing through $z=0$, and we again decompose $C_n$ by cutting it at $z=0$ into $C_{n}^{-}$ followed by $C_{n}^{+}$.  In the exceptional case we require that $C_{n}^{-}\cup\mathbb{R}_+$ with orientation on $\mathbb{R}_+$ continued from $C_{n}^{-}$ is a contour with $p_1,\dots,p_N$ on its left and $p_1^*,\dots,p_N^*$ on its right, and that $C_{n}^{+}$ lies in $\mathbb{C}_+$ after its initial point $z=0$.  We set
\begin{equation}
    E_n:=0,\quad n=1,\dots,N,\quad \text{for $n$ exceptional}.
\label{eq:XS-En-exceptional}
\end{equation}
See Figure~\ref{fig:C0Cm} (center and right panels) and Figure~\ref{fig:Cm>} for illustrations of these contours.

With this setup, we define $\ee^{-\ii L(z)}$ differently on each contour $C_n$ by analytic continuation from $z=\infty$ with initial value $\ee^{-\ii L(\infty)}=1$.  Then a matrix $\mathbf{J}$ of dimension $N\times N$ is defined by 
\begin{equation}
    J_{mn}:=\int_{C_{m}^{-}}\frac{\ee^{-\ii L(z)}-1}{z-p_n}\,\dd z + \int_{C_{m}^{+}}\frac{\ee^{-\ii L(z)}-E_m}{z-p_n}\,\dd z + 2\pi\ii E_m\delta_{m\ge n},\quad m,n=1,\dots,N.
\label{eq:XS-J-def}
\end{equation}
Note that if $m$ is exceptional, then the second integral is on a path from $z=0$ to $z=p_m$ and the integrand is $\ee^{-\ii L(z)}/(z-p_n)$ which is integrable at the terminal endpoint regardless of whether or not $m=n$ because $\ee^{-\ii L(z)}$ vanishes to at least first order there.  Finally, we define $\Delta$ as
\begin{equation}
\Delta:=\sum_{n=1}^N\det(\mathbf{J}\mathop{\longleftarrow}^n\mathbf{d}),\quad\mathbf{d}:=(1-E_1,\dots,1-E_N)^\top
\label{eq:XS-DeltaN}
\end{equation}
where the notation $\mathbf{J}\displaystyle\mathop{\longleftarrow}^n\mathbf{d}$ means the matrix obtained from $\mathbf{J}$ by replacing its $n$th column with $\mathbf{d}$.  We remark that the choice of $z=0$ as the break point for the contours $C_m$ is quite arbitrary in that different choices of break point amount to adding to each column of $\mathbf{J}$ a multiple of the vector $\mathbf{d}$; hence $\Delta$ is invariant.

It is known \cite{Sun2020} that the multisoliton solutions of the Benjamin-Ono equation \eqref{eq:XS-BO} are exactly those rational functions of the form \eqref{eq:XS-u0-rational} for which  $\ii c_1,\dots,\ii c_N$ are all \emph{positive} integers.  In this situation none of the indices are exceptional, and it is straightforward to confirm from \eqref{eq:XS-En-non-exceptional} that $E_n=1$ for all $n=1,\dots,N$.  Therefore $\mathbf{d}=\mathbf{0}$ and consequently $\Delta=0$ for all such reflectionless and hence nongeneric solutions.  On the other hand, in \cite{BlackstoneGGM24a} the solution of \eqref{eq:XS-BO} for $u_0(x)=-2/(1+x^2)$ was analyzed for large $t$.  In the same work, the reflection coefficient was computed for this datum as
\begin{equation}
u_0(x)=\frac{-2}{1+x^2}\implies    \beta(\lambda)=\frac{2\pi\ii\ee^\lambda}{\mathrm{Ei}(2\lambda)-\ii\pi},\quad\lambda>0.
\label{eq:XS-beta-neg-soliton}
\end{equation}
Using the series expansion of $\mathrm{Ei}(2\lambda)$ for small $\lambda$ (see \cite[Equation 6.6.1]{DLMF}) one easily confirms that \eqref{eq:XS-beta-neg-soliton} satisfies \eqref{eq:XS-beta-generic}, so this is a generic potential.  The potential $u_0(x)=-2/(1+x^2)$ is also of the form \eqref{eq:XS-u0-rational} with $N=1$, $p_1=\ii$, and $c_1=\ii$, so the only index $n=1$ is exceptional.  Using  \eqref{eq:XS-En-exceptional} gives $E_1=0$ and hence the vector $\mathbf{d}$ amounts to the scalar $1$.  It follows from \eqref{eq:XS-DeltaN} that for this generic potential, $\Delta=1\neq 0$.    

More generally we have the following.
\begin{theorem}
    Suppose that $u_0$ is given by \eqref{eq:XS-u0-rational} subject to \eqref{eq:XS-u0-in-L1}.  Then $\beta(\lambda)$ has the form
    \begin{equation}
        \beta(\lambda)=\frac{2\pi\ii \varphi_1(\lambda)+\lambda \varphi_2(\lambda)}{\varphi_1(\lambda)\ln(\lambda)+\varphi_3(\lambda)},\quad\lambda>0,
    \label{eq:XS-rational-beta-general}
    \end{equation}
    where $\varphi_1,\varphi_2,\varphi_3$ are functions analytic at $\lambda=0$ and $\varphi_1(0)=\Delta$.  If $\Delta\neq 0$, then $u_0$ is a generic initial datum for the Benjamin-Ono equation \eqref{eq:XS-BO}.
\label{thm:XS-generic}
\end{theorem}
We give the proof in the Appendix.  We do not know whether $\Delta\neq 0$ is a necessary condition for genericity of a rational potential of the form \eqref{eq:XS-u0-rational}.  As a simple generalization of the two cases considered above, we note that if 
$u_0$ is a rational potential of the form \eqref{eq:XS-u0-rational} with $N=1$ and $c_1\in\ii\mathbb{R}\setminus\{0\}$, then $\Delta\neq 0$ and hence $u_0$ is generic \emph{unless} $c_1=-\ii n$ for $n=1,2,3,\dots$ (in which case $u_0$ is a $n$-soliton solution with $\beta(\lambda)\equiv 0$).  Indeed, either $c_1=\ii n$ for $n=1,2,3,\dots$ giving the exceptional case with $E_1=0$ and hence $\Delta=1-E_1=1$, or $c_1\in\ii(\mathbb{R}\setminus\mathbb{Z})$ giving the non-exceptional case with $E_1=\ee^{2\pi c_1}\neq 1$, so $\Delta = 1-E_1\neq 0$.

When one considers rational approximation of initial data, a natural  question is whether anything is lost by assuming the condition $\Delta\neq 0$ sufficient for genericity.  To partially address this question, we have the following result.
\begin{theorem}
    Potentials $u_0$ of the form \eqref{eq:XS-u0-rational} satisfying \eqref{eq:XS-u0-in-L1} and for which $\Delta\neq 0$ are dense in $L^2(\mathbb{R},\mathbb{R})$.
\end{theorem}
\begin{proof}
Let $v\in L^2(\mathbb{R},\mathbb{R})$ and $\epsilon>0$ be given.  We will find a potential $u_0$ with the desired properties so that $\|u_0-v\|_{L^2(\mathbb{R})}<\epsilon$.  

    First note that simple-pole rational functions $u_0$ of the form \eqref{eq:XS-u0-rational} satisfying the condition \eqref{eq:XS-u0-in-L1} and having no exceptional indices are dense in $L^2(\mathbb{R},\mathbb{R})$.  Therefore, there exists such a function $u_0^\epsilon$ for which $\|u^\epsilon_0-v\|_{L^2(\mathbb{R})}<\frac{1}{2}\epsilon$.  If $\Delta\neq 0$ for $u_0^\epsilon$, then we are done because $\|u_0^\epsilon-v\|_{L^2(\mathbb{R})}<\frac{1}{2}\epsilon<\epsilon$.
    
    Otherwise, we will find another simple-pole rational function $\widetilde{u}_0^\epsilon$ of the same type but for which $\Delta\neq 0$ and for which $\|\widetilde{u}_0^\epsilon-u_0^\epsilon\|_{L^2(\mathbb{R})}<\frac{1}{2}\epsilon$ so that $\|\widetilde{u}_0^\epsilon-v\|_{L^2(\mathbb{R})}\le \|\widetilde{u}_0^\epsilon-u_0^\epsilon\|_{L^2(\mathbb{R})}+\|u_0^\epsilon-v\|_{L^2(\mathbb{R})}<\frac{1}{2}\epsilon+\frac{1}{2}\epsilon=\epsilon$.  Fix $N=N(\epsilon)$ as the number of poles in the upper half-plane for $u_0^\epsilon$, and fix also the poles $p_n=p_n^\epsilon$, $n=1,\dots,N$.  We will construct $\widetilde{u}_0^\epsilon$ by perturbing the coefficients $c_n=c_n^\epsilon$ of $u_0^\epsilon$ consistent with the condition \eqref{eq:XS-u0-in-L1}.  Let $\mathbb{H}$ denote the complex hyperplane in $\mathbb{C}^{2N}$ consisting of vectors $\mathbf{v}$ with components satisfying $v_1+v_3+\cdots + v_{2N-1}=0$. The identification $c_n=v_{2n-1}+\ii v_{2n}$ and $c_n^*=v_{2n-1}-\ii v_{2n}$ shows that if $\mathbf{v}$ is real, then $\mathbf{v}\in\mathbb{H}$ is the same as the condition \eqref{eq:XS-u0-in-L1}.  Now allowing $\mathbf{v}\in\mathbb{H}$ to be complex,  we define an entire function $\widetilde{\Delta}:\mathbb{H}\to\mathbb{C}$ by applying the same procedure used to define $\Delta$ except that we define the contours $C_1,\dots,C_N$ and the quantities $E_1,\dots,E_N$ as if all indices $n=1,\dots,N$ were non-exceptional.  Clearly, if $\mathbf{v}$ is real, then $\widetilde{\Delta}=\Delta$ for fixed $p_1,\dots,p_N$, unless there is an exceptional index.  Since no indices are exceptional for $u_0^\epsilon$, which is an open condition on the coefficients $c_1,\dots,c_N$, there is a neighborhood $U$ of the real part $\mathbb{H}_\mathbb{R}$ of $\mathbb{H}$ containing $\mathbf{v}_0:=(\mathrm{Re}(c_1^\epsilon),\mathrm{Im}(c_1^\epsilon),\dots,\mathrm{Re}(c_N^\epsilon),\mathrm{Im}(c_N^\epsilon))^\top$ on which $\Delta=\widetilde{\Delta}$ is real-analytic.
    
    We claim that $\widetilde{\Delta}$ does not vanish identically on $\mathbb{H}$.  To see this, we choose a point of the form $\mathbf{v}=(0,v_2,0,\dots,0)^\top\in\mathbb{H}$ (or $c_2=\cdots=c_N=0$ while $c_1=\ii v_2$ and $c_1^*=-\ii v_2$), so that $\ee^{-\ii L(z)}$ is the branch of $\ee^{-\ii L(z)}=[(z-p_1)/(z-p_1^*)]^{v_2}$ that tends to $1$ as $z\to\infty$ backward along each contour $C_1,\dots,C_N$. Then, $E_1=E_2=\cdots=E_N=\ee^{2\pi\ii v_2}$ is the (common) corresponding limiting value along each contour in the forward direction, so all elements of the vector $\mathbf{d}$ are the same:  $d_m=1-E_m=1-\ee^{2\pi\ii v_2}$, $m=1,\dots,N$.  In the matrix elements of $\mathbf{J}$, the initial contour arcs $C_m^-$ may all be taken to be the same (denoted $C^-$), while a residue computation shows that
    \begin{equation}
        \int_{C_m^+}\frac{\ee^{-\ii L(z)}-\ee^{2\pi\ii v_2}}{z-p_n}\,\dd z = \int_{C_1^+}\frac{\ee^{-\ii L(z)}-\ee^{2\pi\ii v_2}}{z-p_n}\,\dd z + 2\pi\ii (\ee^{-\ii L(p_n)}-\ee^{2\pi\ii v_2})\delta_{1<n\le m},\quad m\ge 2.
    \end{equation}
    Therefore, setting 
    \begin{equation}
        f_n:=\int_{C^-}\frac{\ee^{-\ii L(z)}-1}{z-p_n}\,\dd z +\int_{C_1^+}\frac{\ee^{-\ii L(z)}-\ee^{2\pi\ii v_2}}{z-p_n}\,\dd z,\quad n=1,\dots,N,
    \end{equation}
    we see that the first row elements of $\mathbf{J}$ are
    $J_{1n}=f_n+2\pi\ii\ee^{2\pi\ii v_2}\delta_{n=1}$ while the subsequent rows have elements that can be written in the form $J_{mn}=J_{1n}+2\pi\ii\ee^{-\ii L(p_n)}\delta_{1<n\le m}$ for $m=2,\dots,N$.  Therefore, by row operations subtracting the first row from each subsequent row,
    \begin{multline}
        \det(\mathbf{J}\mathop{\longleftarrow}^n\mathbf{d}) \\= 
        \det\left(\begin{bmatrix}f_1+2\pi\ii\ee^{2\pi\ii v_2} & f_2 & f_3 & 
        \cdots & f_N\\        
        0 & 2\pi\ii\ee^{-\ii L(p_2)} & 0 & 
        \cdots & 0\\        
        0 & 2\pi\ii\ee^{-\ii L(p_2)} & 2\pi\ii\ee^{-\ii L(p_3)} & 
        \ddots & \vdots\\           
        \vdots & \vdots & \vdots &
        \ddots & 0\\        
        0 & 2\pi\ii\ee^{-\ii L(p_2)}& 2\pi\ii\ee^{-\ii L(p_3)} & 
        \cdots &2\pi\ii \ee^{-\ii L(p_N)}\end{bmatrix}\mathop{\longleftarrow}^n\begin{bmatrix}1-\ee^{2\pi\ii v_2} \\0\\0\\
        \vdots\\ 0\end{bmatrix}\right).
    \end{multline}
    This determinant obviously vanishes unless $n=1$, in which case it is easily evaluated.  We deduce that
    \begin{equation}
        \widetilde{\Delta}(0,v_2,0,\dots,0)=
    \det(\mathbf{J}\displaystyle\mathop{\longleftarrow}^1\mathbf{d}) = (1-\ee^{2\pi\ii v_2})(2\pi\ii)^{N-1}\prod_{n=2}^N\ee^{-\ii L(p_n)},
    \end{equation}
    which only vanishes if $v_2\in\mathbb{Z}$.  

    Returning to the neighborhood $U\subset\mathbb{H}_\mathbb{R}$, the real and imaginary parts of $\Delta=\widetilde{\Delta}$ are both real-analytic functions on $U$.  Therefore $m:=|\Delta|^2=\mathrm{Re}(\Delta)^2+\mathrm{Im}(\Delta)^2$ is a real-analytic function on $U$ that does not vanish identically, although $m=0$ for $\mathbf{v}=\mathbf{v}_0\in U$.  We introduce explicit coordinates $\mathbf{x}=(x_1,\dots,x_{2N-1})^\top$ on $U$ with $\mathbf{v}=\mathbf{v}_0$ corresponding to $\mathbf{x}=\mathbf{0}$ by 
    \begin{equation}
       \mathbf{v}=\mathbf{v}(\mathbf{x}):=\mathbf{v}_0+\left(-\sum_{j=1}^{N-1}x_{2j},x_1,x_2,\dots,x_{2N-1}\right)^\top 
    \end{equation}
    and we write $m=m(\mathbf{x})$.  Let $k\ge 1$ be the smallest integer such that $\partial_\mathbf{x}^\alpha m(\mathbf{0})=0$ for all multi-indices $\alpha$ with $|\alpha|<k$.  Note that $k<\infty$ because $m(\mathbf{x})$ does not vanish identically.  Let $P_k(\mathbf{x})$ denote the (nonvanishing) homogeneous polynomial consisting of the terms of degree $k$ in the Taylor expansion of $m(\mathbf{x})$ about $\mathbf{x}=\mathbf{0}$.  Then there is a unit vector $\mathbf{e}\in\mathbb{R}^{2N-1}$ such that $P_k(\mathbf{e})\neq 0$.  For $\rho>0$, we then have $m(\rho\mathbf{e})=P_k(\rho\mathbf{e})+O(\rho^{k+1})=\rho^kP_k(\mathbf{e}) + O(\rho^{k+1})$ as $\rho\downarrow 0$.  In particular, $m(\rho\mathbf{e})\neq 0$ for $\rho>0$ sufficiently small.  Since for fixed $N$ and $p_1,\dots,p_N$, the rational form \eqref{eq:XS-u0-rational} defines a continuous map from the coefficients $c_1,\dots,c_N$ into $L^2(\mathbb{R})$, choosing $\widetilde{\mathbf{v}}=\mathbf{v}(\mathbf{x})=\mathbf{v}(\rho\mathbf{e})$ for $\rho>0$ sufficiently small gives the coefficients $\widetilde{c}_n=\widetilde{v}_{2n-1}+\ii \widetilde{v}_{2n}$ of a rational function $\widetilde{u}_0^\epsilon$ with the desired properties, and the proof is complete.
\end{proof}

The main result of this paper is the following theorem, which (i) shows that for a wide class of rational initial data, the self-similar approximation is not quite correct in that the rate of decay has to be changed from $t^{-1/2}$ to $t^{-1/2}/\ln(t)$ and (ii) yields a completely explicit expression for the limiting profile $U(\xi)$.
\begin{theorem}
    Fix distinct points $p_1,\dots,p_N$ in the open upper half-plane and complex constants $c_1,\dots,c_N$ satisfying \eqref{eq:XS-u0-in-L1}.  Let $u(t,x)$ be the solution of the Benjamin-Ono equation \eqref{eq:XS-BO} with rational initial data $u_0(x)=u(0,x)$ of the form \eqref{eq:XS-u0-rational}.  
    If also $\Delta\neq 0$, then
    \begin{equation}
        \lim_{t\to+\infty}t^{1/2}\ln(t)u(t,2t^{1/2}\xi) = U(\xi)
    \end{equation}
    in the $L^\infty_\mathrm{loc}(\mathbb{R})$ sense with respect to $\xi$, where
    \begin{equation}
        U(\xi):=F(\xi)+F(\xi)^*=2\mathrm{Re}(F(\xi)),
    \label{eq:U-F-intro}
    \end{equation}
    and, with $L_+$ denoting a simple contour from $\infty\ee^{3\pi\ii/4}$ to $\infty\ee^{-\ii\pi/4}$ that passes above the origin,
    \begin{equation}
        F(\xi):=\frac{\ee^{-\ii\pi/4}}{\sqrt{\pi}}\int_{L_+}\ee^{-\ii(w-\xi)^2}\frac{\dd w}{w}.
    \label{eq:F-intro}
    \end{equation}
\label{thm:XS-main}
\end{theorem}
Although it decays with a well-defined universal profile $U(\xi)$ and $\xi=x/(2t^{1/2})$  is indeed the similarity variable for the Benjamin-Ono equation, the leading term of $u(t,x)$ is evidently \emph{not} a self-similar solution of the Benjamin-Ono equation, because it comes with a faster rate of decay, $t^{-1/2}/\ln(t)$ rather than the self-similar rate of $t^{-1/2}$.

The proof of Theorem~\ref{thm:XS-main} is based on an explicit formula for the solution of the Cauchy problem found by one of the authors in \cite{Gerard22}, combined with an implied expression valid for rational initial data that was obtained in \cite{BlackstoneGGM24a}.

\subsection{Properties of $F(\xi)$ and $U(\xi)$}
It is clear from the formula \eqref{eq:F-intro} that $\xi\mapsto F(\xi)$ is an entire function.  Its asymptotic behavior as $\xi\to\infty$ in $\mathbb{C}$ can be analyzed by the method of steepest descent.  
One can therefore see that 
\begin{equation}
F(\xi)=\frac{1}{\ii\xi} + O(\xi^{-3}),\quad|\xi|\to\infty,\quad 0\le\arg(\xi)<\pi,
\label{eq:XS-f-right}
\end{equation}
and accounting for a residue at the origin,
\begin{equation}
F(\xi)=2\sqrt{\pi}\ee^{-3\pi\ii/4}\ee^{-\ii\xi^2} +\frac{1}{\ii\xi} + O(\xi^{-3}),\quad\xi\to -\infty.
\label{eq:XS-f-left}
\end{equation}
In particular, $F$ is bounded in the closed upper half-plane, and thus is an element of the Hardy-type space $L_+^\infty(\mathbb{R})$.
Also, observe that $F$ satisfies a first-order linear differential equation:
\begin{equation}
\begin{split}
    F'(\xi)& = -2\ii\frac{\ee^{-\ii\pi/4}}{\sqrt{\pi}}\int_{L_+}\ee^{-\ii (w-\xi)^2}\frac{\xi-w}{w}\,\dd w \\ &= -2\ii\xi F(\xi)+\frac{2\ee^{\ii\pi/4}}{\sqrt{\pi}}\int_{L_+}\ee^{-\ii (w-\xi)^2}\,\dd w \\ &= -2\ii\xi F(\xi) + 2.
\end{split}
\label{eq:fprime}
\end{equation}
The unique solution vanishing in the limit $\xi\to+\infty$ (according to \eqref{eq:XS-f-right}) is
given explicitly in terms of a convergent improper Fresnel-type integral:
\begin{equation}
    F(\xi)=-2\ee^{-\ii\xi^2}\int_\xi^{+\infty}\ee^{\ii w^2}\,\dd w \xRightarrow[\eqref{eq:U-F-intro}]{} U(\xi)=-4\int_\xi^{+\infty}\cos(w^2-\xi^2)\,\dd w.
\label{eq:F-Fresnel}
\end{equation}

Using \eqref{eq:U-F-intro}, the asymptotics
\eqref{eq:XS-f-right}--\eqref{eq:XS-f-left} imply that 
\begin{equation}
U(\xi)=\begin{cases}O(\xi^{-3}),& \xi\to+\infty   \\
4\sqrt{\pi}\cos(\xi^2+3\pi/4)+O(\xi^{-3}),& \xi\to-\infty,
\end{cases}
\end{equation}
because the terms proportional to $\xi^{-1}$ do not contribute to $U(\xi)$ in either limit. 
See Figure~\ref{fig:U-of-xi} for a plot of $U(\xi)$.
\begin{figure}[h]
\begin{center}
    \includegraphics[width=0.4\linewidth]{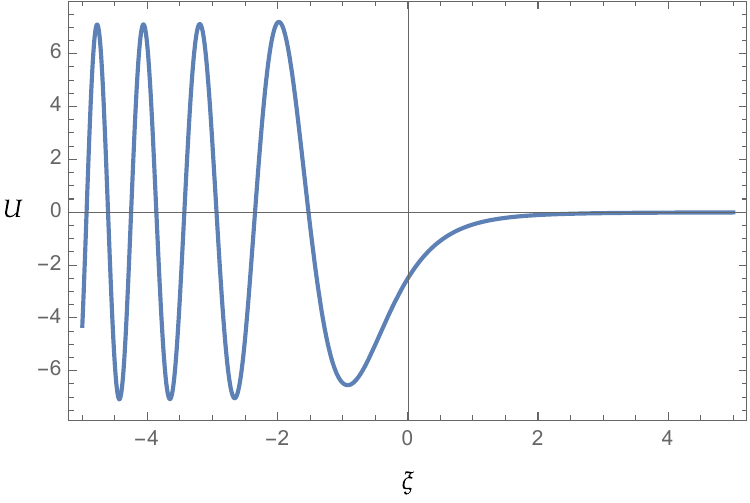}
\end{center}
\caption{The universal profile function $U(\xi)$.}
\label{fig:U-of-xi}
\end{figure}

For a function $f\in H^1(\mathbb{R})$, the nonlocal operator $|D|$ in the Benjamin-Ono equation is defined by 
\begin{equation}
    |D|f(\xi):=-\ii\partial_\xi (\Pi_+ f(\xi)-\Pi_-f(\xi))=-\ii (\Pi_+ f'(\xi)-\Pi_-f'(\xi)),
\end{equation}
where $\Pi_\pm$ are the complementary and orthogonal Cauchy-Szeg\H{o} projectors from $L^2(\mathbb{R})$ onto the Hardy spaces $L^2_\pm(\mathbb{R})$.  We sometimes use the shorthand $\Pi:=\Pi_+$.  The function $f$ can itself be written as a sum of the projections:  $f(\xi)=\Pi_+f(\xi)+\Pi_-f(\xi)$.  Now suppose that $g\in L^\infty(\mathbb{R})\cap C^1(\mathbb{R})$ and that it is known that there exists a function $g_+\in L^\infty_+(\mathbb{R})\cap C^1(\mathbb{R})$ such that $g(\xi)=g_+(\xi)+g_+(\xi)^*$. We may define $|D|g$ as follows.  First, we show that $g_+$ is uniquely determined from $g$ up to a constant.  Indeed, if $g_+(\xi)+g_+(\xi)^*=\tilde{g}_+(\xi)+\tilde{g}_+(\xi)^*$ with $g_+,\tilde{g}_+\in L^\infty_+(\mathbb{R})\cap C^1(\mathbb{R})$, then $\tilde{g}_+-g_+\in L^\infty_+(\mathbb{R})\cap C^1(\mathbb{R})$ is continuous and purely imaginary, and hence by Schwarz reflection extends to an entire function that is bounded.  By Liouville's theorem, $\tilde{g}_+-g_+$ is a purely imaginary constant function.  Let $g_+$ denote any element of the corresponding equivalence class.
Next, we introduce a cutoff function by
\begin{equation}
\chi_\epsilon(\xi):=\frac{1}{1+\ii\epsilon\xi},\quad\epsilon>0.
\end{equation}
Note that for each $\epsilon>0$, $\chi_\epsilon\in L^2_+(\mathbb{R})$, and that $\chi_\epsilon\to 1$ in $L^\infty_\mathrm{loc}(\mathbb{R})$ as $\epsilon\downarrow 0$.  Therefore, for $\epsilon>0$, we have $\chi_\epsilon g_+\in L^2_+(\mathbb{R})\cap L^\infty_+(\mathbb{R})$, and $\chi_\epsilon g_+\to g_+$ in $L^\infty_\mathrm{loc}(\mathbb{R})$.  
Then, we consider, for each $\epsilon>0$, $g_\epsilon$ defined by
\begin{equation}
    g_\epsilon:=\chi_\epsilon g_+ + \chi_\epsilon^*g_+^*,
\end{equation}
for which $\Pi_+g_\epsilon = \chi_\epsilon g_+$ and $\Pi_-g_\epsilon = \chi_\epsilon^*g_+^*$.
Then, for each $\epsilon>0$ the action $|D|g_\epsilon$ makes sense and is given by
\begin{equation}
\begin{aligned}
    |D|g_\epsilon(\xi) &= -\ii\partial_\xi(\chi_\epsilon(\xi) g_+(\xi) -\chi_\epsilon(\xi)^*g_+(\xi)^*)\\
    &=-\epsilon\chi_\epsilon(\xi)^2g_+(\xi)-\epsilon\chi_\epsilon(\xi)^{*2}g_+(\xi)^*-\ii(\chi_\epsilon(\xi)g_+'(\xi)-\chi_\epsilon(\xi)^*g_+'(\xi)^*).
\end{aligned}
\end{equation}
Under the conditions in force the right-hand side has a limit in $L^\infty_\mathrm{loc}(\mathbb{R})$ as $\epsilon\downarrow 0$, and we define
\begin{equation}
    |D|g(\xi):=\lim_{\epsilon\downarrow 0}|D|g_\epsilon(\xi) = -\ii (g_+'(\xi)-g_+'(\xi)^*).
\label{eq:AbsD}
\end{equation}
Note that the definition is independent of the choice of equivalence class representative for $g_+$.

\begin{theorem}
    The function $U(\xi):=F(\xi)+F(\xi)^*$ with $F\in L^\infty_+(\mathbb{R})\cap C^1(\mathbb{R})$ defined by \eqref{eq:F-intro} satisfies
    \begin{equation}
        -\xi U(\xi)=\frac{1}{2}|D|U(\xi).
        \label{eq:linearized}
    \end{equation}
\label{thm:XS-linearized}
\end{theorem}
\begin{proof}
    Applying \eqref{eq:AbsD} with $g=U$ and $g_+=F$ gives
    \begin{equation}
        \frac{1}{2}|D|U(\xi)=-\ii\left(\frac{1}{2}F'(\xi)-\frac{1}{2}F'(\xi)^*\right).
    \end{equation}
    Then, using the differential equation \eqref{eq:fprime},
    \begin{equation}
        \frac{1}{2}|D|U(\xi)=-\ii\left([-\ii\xi F(\xi)+1]-[\ii\xi F'(\xi)^*+1]\right) = -\xi(F(\xi)+F(\xi)^*) = -\xi U(\xi),
    \end{equation}
    so the proof is finished.
\end{proof}
Theorem~\ref{thm:XS-linearized} shows that the profile $U(\xi)$ is a solution of \eqref{eq:linearized} rather than the self-similar profile equation \eqref{eq:integrated-BO-profile}.  Noting that the former can be obtained from the latter by omitting the nonlinear term $2U(\xi)^2$, we have the heuristic interpretation that the linearization arises because, according to Theorem~\ref{thm:XS-main}, $t^{1/2}u(t,2t^{1/2}\xi)$ is an increasingly small function of $\xi$ (decaying in $L^\infty_\mathrm{loc}$ as $1/\ln(t)$), so $U(\xi)$ should be an infinitesimal solution of the nonlinear profile equation \eqref{eq:integrated-BO-profile}.

Without the hypotheses of Theorem~\ref{thm:XS-main}, the asymptotic behavior of $u(t,2t^{1/2}\xi)$ can be quite different.  Indeed, if $u(t,x)$ is a $n$-soliton solution, then $\Delta=0$ and since the soliton with speed $v>0$ decays as $O((x-vt)^{-2})$ we find that $u(t,2t^{1/2}\xi)=O(t^{-2})$ as $t\to+\infty$ uniformly for bounded $\xi$.  This is a substantially faster rate of decay.  On the other hand, there are examples showing that if the condition \eqref{eq:XS-u0-in-L1} is violated, then there can be infinitely many Lax eigenvalues $\lambda_n<0$ corresponding to infinitely many solitons with speeds $v_n>0$ accumulating at zero (see, for instance, \cite{GassotG26}).  One therefore might expect that the solution should be larger than predicted by Theorem~\ref{thm:XS-main} and indeed one can prove that in some examples $t^{1/2}u(t,2t^{1/2}\xi)$ converges as $t\to+\infty$ to a non-universal nonzero limiting profile \cite{GassotGM26+}.  The decay rate of such a solution is slower by a logarithmic factor than that predicted by Theorem~\ref{thm:XS-main}.

That said, a reasonable question is whether an analogue of Theorem~\ref{thm:XS-main} holds if only the assumption of rationality is dropped.  A generalization of the condition \eqref{eq:XS-u0-in-L1} would be an assumption that $u_0(\diamond)$ and $\diamond u_0(\diamond)$ lie in $L^2(\mathbb{R})$, and a natural analogue of the condition $\Delta\neq 0$ would simply be the assumption that $u_0$ is generic.  While a proof might still be based on the explicit formula of \cite{Gerard22}, it would necessarily be quite different without the rational assumption.  

\subsection{Further remarks}
The universal profile function $U(\xi)$ may be viewed as an analogue for the BO equation of the Miura transform $U_\mathrm{HM}(\xi):=V_\mathrm{HM}(\xi)^2+V_\mathrm{HM}'(\xi)$ of the Hastings-McLeod solution $V_\mathrm{HM}(\xi)$ of the Painlev\'e-II equation for the KdV equation.  Both are specific functions that yield the asymptotic behavior for large $t$ and $x/t$ sufficiently small for a wide class of initial data.  However, unlike $U_\mathrm{HM}(\xi)$, $U(\xi)$ is easily expressed in terms of elementary functions.  This observation continues a theme that was also observed recently in connection with universality in the small-dispersion limit \cite{BlackstoneMM26}.

In the large-time behavior of the KdV equation,  Hastings-McLeod universality for $x/t=O(t^{-1/3})$ is always accompanied by a collisionless shock region to the left of the central self-similar region.  This was originally suggested \cite{AblowitzS77} by a mismatch of the self-similar asymptotic formula with the leading behavior in the dispersive decay region $x/t<-\delta$.  Similar reasoning suggests that there is no analogue of the collisionless shock region for the BO equation.
Indeed, the large-$t$ formula for $u$ given in Theorem~\ref{thm:XS-main} agrees asymptotically with the large-$t$ formula given in \cite{BlackstoneGGM24a} in an overlap domain corresponding to large negative $\xi$ in \eqref{eq:XS-u-convergence} and small negative $y$ in \cite[Equation (1.19)]{BlackstoneGGM24a}.  Of course this observation does not amount to a proof because the two regions of validity do not obviously overlap.

\subsection{Outline of the paper}
In Section~\ref{sec:XS-minus-a-soliton} we first prove Theorem~\ref{thm:XS-main} in a simple special case, namely the generic initial datum $u_0(x)=-2/(1+x^2)$, for which the long-time asymptotic behavior when $|x/t|\ge\delta>0$ was proved in \cite{BlackstoneGGM24a}.  Here we follow the approach of that paper (rather than the more general approach omitting the rational assumption altogether given in \cite{GassotGM26}) but focus instead on the regime $|x/t|\lesssim t^{-1/2}$, i.e., where $|x|=O(\sqrt{t})$.  Then in Section~\ref{sec:XS-general} we prove Theorem~\ref{thm:XS-main} in all generality.  Finally, the proof of Theorem~\ref{thm:XS-generic} can be found in the Appendix.

\subsection*{Acknowledgements} 
L. Gassot was supported by the France 2030 framework program, the Centre Henri Lebesgue ANR-11-LABX-0020-01, and  the French Agence Nationale de la Recherche under the ANR project HEAD--ANR-24-CE40-3260.
P. Gérard was partially supported by the French Agence Nationale de la Recherche under the ANR project ISAAC--ANR-23--CE40-0015-01.
 P. D. Miller was partially supported by the National Science Foundation under grant DMS-2508694.

\section{Proof of Theorem~\ref{thm:XS-main}:  elementary example}
\label{sec:XS-minus-a-soliton}
Let us first give a quick proof of Theorem~\ref{thm:XS-main} in a special case.
We analyze the solution $u(t,x)$ of \eqref{eq:XS-BO} with initial data $u_0(x)=-2/(1+x^2)$ (the negative of a soliton profile), which following \cite[Section 6]{BlackstoneGGM24a} can be written in the form $u(t,x)=\Pi u(t,x)+\Pi u(t,x)^*$ with
\begin{equation}
    \Pi u(t,x)=\frac{N(t,x)}{D(t,x)},
    \label{eq:XS-Piu}
\end{equation}
where
\begin{equation}
    N(t,x)=-2\begin{vmatrix}J_0(t,x) & I_0(t,x)\\J_1(t,x) & I_1(t,x)\end{vmatrix},\quad D(t,x)=\begin{vmatrix}K_0(t,x) & I_0(t,x)\\K_1(t,x) & I_1(t,x)\end{vmatrix}
\label{eq:XS-NandD}
\end{equation}
and for $p=0,1$,
\begin{equation}
\begin{gathered}
    I_p(t,x):=\int_{C_p}\ee^{-\ii (z-x)^2/(4t)}\frac{\dd z}{z+\ii},\quad
    J_p(t,x):=\int_{C_p}\ee^{-\ii (z-x)^2/(4t)}\frac{\dd z}{(z+\ii)^2},\\
    K_p(t,x):=\int_{C_p}\ee^{-\ii (z-x)^2/(4t)}\,\dd z.
\end{gathered}
\end{equation}
Here $C_0$ and $C_1$ are contours that can be parametrized by $z=\ii + \ee^{-\ii\pi/4}\zeta$ with $-\infty<\zeta<\infty$ and $-\infty<\zeta<0$ respectively.

We now study $u(t,x)$ with $x=2t^{1/2}\xi$ for bounded $\xi$ in the limit $t\to+\infty$ by studying the integrals $I_p$, $J_p$, and $K_p$ for $p=0,1$.  As pointed out in \cite[Section 6]{BlackstoneGGM24a}, $K_0$ is a Gaussian integral independent of $x$ given by 
\begin{equation}
    K_0(t,x)=\ee^{-\ii\pi/4}2\sqrt{\pi t},\quad t>0.
\end{equation}
To find the asymptotic behavior of the other five integrals when $x=2t^{1/2}\xi$ for $\xi=O(1)$, we make 
the corresponding scaling $z=2t^{1/2}w$, 
and hence obtain
\begin{equation}
\begin{split}
    I_p(t,2t^{1/2}\xi)&=\int_{C_p/(2t^{1/2})}\ee^{-\ii (w-\xi)^2}\frac{\dd w}{w+\ii/(2t^{1/2})},\\ J_p(t,2t^{1/2}\xi)&=\frac{1}{2t^{1/2}}\int_{C_p/(2t^{1/2})}\ee^{-\ii (w-\xi)^2}\frac{\dd w}{(w+\ii/(2t^{1/2}))^2},\\
K_p(t,2t^{1/2}\xi)&=2t^{1/2}\int_{C_p/(2t^{1/2})}\ee^{-\ii (w-\xi)^2}\,\dd w.
\end{split}
\end{equation}
In the case of $I_0(t,2t^{1/2}\xi)$ and $J_0(t,2t^{1/2}\xi)$, we can use Cauchy's theorem to replace the integration contour with one that is independent of $t>0$; any simple contour $L_+$ from $\infty\ee^{3\pi\ii/4}$ to $\infty\ee^{-\ii\pi/4}$ that passes above the origin will do.  Hence one sees that
\begin{equation}
    I_0(t,2t^{1/2}\xi)=\int_{L_+}\ee^{-\ii (w-\xi)^2}{\frac{\dd w}{w}} + O(t^{-1/2}),\quad J_0(t,2t^{1/2}\xi)=\frac{1}{2t^{1/2}}\int_{L_+}\ee^{-\ii (w-\xi)^2}\frac{\dd w}{w^2} + O(t^{-1}),
\label{eq:I0J0}
\end{equation}
in the limit $t\to+\infty$, uniformly for $\xi=O(1)$.  For the integrals over $C_1$, we see that the rescaled contour $C_1/(2t^{1/2})$ terminates at $w=\ii/(2t^{1/2})$.  Thus $K_1(t,2t^{1/2}\xi)$ can be expanded as
\begin{equation}
K_1(t,2t^{1/2}\xi)=2t^{1/2}\int_{\ee^{3\pi\ii/4}\infty}^0\ee^{-\ii (w-\xi)^2}\,\dd w+O(1),
\end{equation}
as $t\to+\infty$ with $\xi=O(1)$.  For $I_1(t,2t^{1/2}\xi)$ and $J_1(t,2t^{1/2}\xi)$ there is a singularity near the integration endpoint so a local analysis must be performed.  Let $L_1$ be the diagonal half-line terminating at $w=\ii$ with $\arg(w-\ii)=3\pi/4$.  Then by Cauchy's theorem we can replace $C_1/(2t^{1/2})$ with the contour $L_1\cup [\ii\to \ii/(2t^{1/2})]$ and therefore
\begin{equation}
I_1(t,2t^{1/2}\xi)=\int_{L_1}\frac{\ee^{-\ii (w-\xi)^2}\,\dd w}{w+\ii/(2t^{1/2})} + \int_\ii^{\ii/(2t^{1/2})}\frac{\ee^{-\ii (w-\xi)^2}\,\dd w}{w+\ii/(2t^{1/2})}.
\end{equation}
The first term has a limit as $t\to+\infty$, so is $O(1)$.  The second term can be written as
\begin{multline}
    \int_\ii^{\ii/(2t^{1/2})}\frac{\ee^{-\ii (w-\xi)^2}\,\dd w}{w+\ii/(2t^{1/2})} = \int_\ii^{\ii/(2t^{1/2})}\frac{\ee^{-\ii (w-\xi)^2}-\ee^{-\ii (-\ii/(2t^{1/2})-\xi)^2}}{w+\ii/(2t^{1/2})}\,\dd w \\+ \ee^{-\ii (-\ii/(2t^{1/2})-\xi)^2}\int_\ii^{\ii/(2t^{1/2})}\frac{\dd w}{w+\ii/(2t^{1/2})}
\label{eq:second-term}
\end{multline}
and again the first term on the right-hand side is $O(1)$ since the singularity at $w=-\ii/(2t^{1/2})$ has been cancelled.  By explicit calculation, the integral factor in the second term on the right-hand side is $-\frac{1}{2}\ln(t)+O(t^{-1/2})$, so expanding the exponential factor we obtain
\begin{equation}
    I_1(t,2t^{1/2}\xi)=-\frac{1}{2}\ee^{-\ii\xi^2}\ln(t) + O(1)
\end{equation}
as $t\to+\infty$ for $\xi=O(1)$. For $J_1(t,2t^{1/2}\xi)$, we similarly split up
\begin{equation}
2t^{1/2}J_1(t,2t^{1/2}\xi)=\int_{L_1}\frac{\ee^{-\ii (w-\xi)^2}\,\dd w}{(w+\ii/(2t^{1/2}))^2} + \int_\ii^{\ii/(2t^{1/2})}\frac{\ee^{-\ii (w-\xi)^2}\,\dd w}{(w+\ii/(2t^{1/2}))^2},
\end{equation}
and again the first term on the right-hand side is $O(1)$.  By the same grouping of the exponential terms in the numerator as in \eqref{eq:second-term}, we can obtain
\begin{equation}
\begin{aligned}
    \int_\ii^{\ii/(2t^{1/2})}\frac{\ee^{-\ii (w-\xi)^2}\,\dd w}{(w+\ii/(2t^{1/2}))^2}&=\ee^{-\ii\xi^2}\int_\ii^{\ii/(2t^{1/2})}\frac{\dd w}{(w+\ii/(2t^{1/2}))^2} + O(\ln(t))\\
    &=t^{1/2}\ii\ee^{-\ii\xi^2} + O(\ln(t))
\end{aligned}
\end{equation}
so that
\begin{equation}
    J_1(t,2t^{1/2}\xi)=\frac{1}{2}\ii\ee^{-\ii\xi^2} + O(t^{-1/2}\ln(t))
\end{equation}
as $t\to+\infty$ with $\xi=O(1)$.  

Returning now to \eqref{eq:XS-NandD}, 
\begin{equation}
\begin{aligned}
N(t,2t^{1/2}\xi)&=-2\left[J_0(t,2t^{1/2}\xi)I_1(t,2t^{1/2}\xi)-J_1(t,2t^{1/2}\xi)I_0(t,2t^{1/2}\xi)\right]\\&=\ii\ee^{-\ii\xi^2}\int_{L_+}\ee^{-\ii (w-\xi)^2}\frac{\dd w}{w}+O(t^{-1/2}\ln(t)),
\end{aligned}
\end{equation}
and
\begin{equation}
\begin{aligned}
D(t,2t^{1/2}\xi)&=K_0(t,2t^{1/2}\xi)I_1(t,2t^{1/2}\xi)-K_1(t,2t^{1/2}\xi)I_0(t,2t^{1/2}\xi)\\&=-\ee^{-\ii\pi/4}\sqrt{\pi t}\ee^{-\ii\xi^2}\ln(t) + O(t^{1/2}),
\end{aligned}
\end{equation}
so from \eqref{eq:XS-Piu},
\begin{equation}
    \Pi u(t,2t^{1/2}\xi)=\frac{F(\xi)}{t^{1/2}\ln(t)} + O\left(\frac{1}{t^{1/2}\ln(t)^2}\right),
\end{equation}
as $t\to+\infty$ with $\xi=O(1)$, where $F(\xi)$ is the function defined in \eqref{eq:F-intro}.
Going back to $u(t,2t^{1/2}\xi)$ gives
\begin{equation}
    u(t,2t^{1/2}\xi)=\Pi u(t,2t^{1/2}\xi) + \Pi u(t,2t^{1/2}\xi)^* = \frac{U(\xi)}{t^{1/2}\ln(t)} + O\left(\frac{1}{t^{1/2}\ln(t)^2}\right)
\label{eq:XS-u-convergence}
\end{equation}
as $t\to+\infty$ with $\xi=O(1)$, where $U(\xi)$ is defined in terms of $F(\xi)$ in \eqref{eq:U-F-intro}.  Since $u_0(x)=-2/(1+x^2)$ is of the form \eqref{eq:XS-u0-rational} for $N=1$ with data $p_1=\ii$ and $c_1=\ii$ satisfying \eqref{eq:XS-u0-in-L1}, and as pointed out in the introduction it follows that $\Delta=1\neq 0$, this proves Theorem~\ref{thm:XS-main} for this initial datum.
Note that the $O(1/\ln(t))$ relative error of the leading-term approximation in \eqref{eq:XS-u-convergence} decays to zero very slowly.
The convergence established in \eqref{eq:XS-u-convergence} is illustrated in Figure~\ref{fig:XS-u-convergence}.
\begin{figure}[h]
\begin{center}
    \includegraphics[width=0.6\linewidth]{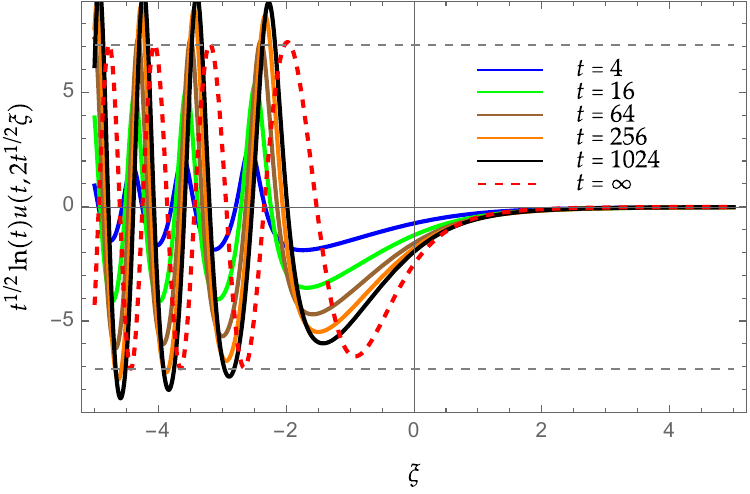}
\end{center}
\caption{A comparison of $t^{1/2}\ln(t)u(t,2t^{1/2}\xi)$ computed via numerical evaluation of the contour integrals in the determinants $N$ and $D$ for various values of $t>0$ with the limiting profile $U(\xi)$ (dashed red curve).}
\label{fig:XS-u-convergence}
\end{figure}

\section{Proof of Theorem~\ref{thm:XS-main}: general case}
\label{sec:XS-general}
\subsection{Solution formula for rational initial data}
We now consider $u_0(x)$ of the general form \eqref{eq:XS-u0-rational} subject to \eqref{eq:XS-u0-in-L1}.  For our present purposes we assume that the contours $C_1,\dots,C_N$ are unbounded only in the direction $\arg(z)=3\pi/4$.  We also introduce an additional unbounded oriented Jordan contour $C_0$ from $z=\infty\ee^{3\pi\ii/4}$ to $z=\infty\ee^{-\ii\pi/4}$ with the poles $p_1,\dots,p_N$ on its left and the poles $p_1^*,\dots,p_N^*$ on its right.  See Figure~\ref{fig:C0Cm}.
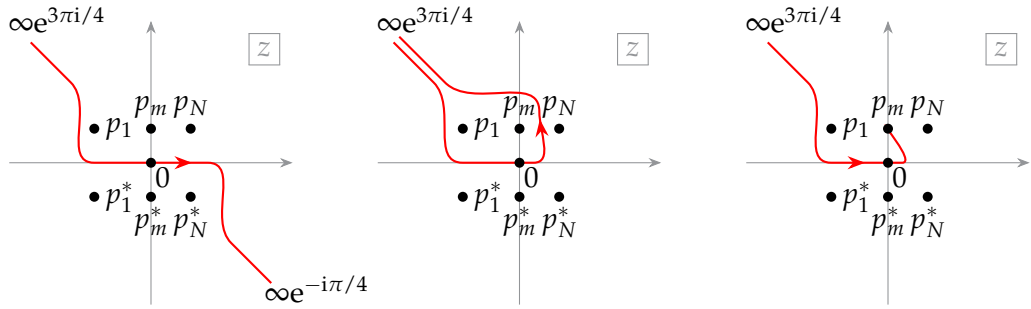
\begin{figure}[h]
\begin{center}
    \begin{tikzpicture}[>=Stealth, scale=0.75]
  \draw[gray, thin, ->] (-2.5,0) -- (2.5,0);
  \draw[gray, thin, ->] (0,-2.5) -- (0,2.5);
  \draw[
  red, thick, decoration={markings, mark=at position 0.6 with {\arrow[red,thick]{Stealth}}},
  postaction={decorate}
  ]
  (-2.121,2.121) to[out=-45, in=135, looseness=1] (-1.414,1.414) to[out=-45, in=180, looseness=1] (-1,0) to[out=0, in=180, looseness=1] (1,0) to[out=0, in=135, looseness=1] (1.414,-1.414) to[out=-45, in=135, looseness=1] (2.121,-2.121); 
  \fill (0,0) circle (2.5pt);
  \node[below right] at (-0.1,0.1) {$0$};

  \fill (-1,0.6) circle (2.5pt);
  \node[right] at (-1,0.6) {$p_1$};

  \fill (0,0.6) circle (2.5pt);
  \node[above] at (0,0.6) {$p_m$};

  \fill (0.7,0.6) circle (2.5pt);
  \node[above] at (0.7,0.6) {$p_N$};

  \fill (-1,-0.6) circle (2.5pt);
  \node[right] at (-1,-0.6) {$p_1^*$};

  \fill (0,-0.6) circle (2.5pt);
  \node[below] at (0,-0.6) {$p_m^*$};

  \fill (0.7,-0.6) circle (2.5pt);
  \node[below] at (0.7,-0.6) {$p_N^*$};

\node[above, left] at (-2.121+1.5,2.121+0.4) {$\infty\ee^{3\pi\ii/4}$};

\node[below, right] at (2.121-0.3,-2.121-0.1) {$\infty\ee^{-\ii\pi/4}$};

  \draw[gray] node at (2,2) {$\boxed{z}$};

\def\shift{6.5}
  \draw[gray, thin, ->] ({-2.5+\shift},0) -- ({2.5+\shift},0);
  \draw[gray, thin, ->] ({0+\shift},-2.5) -- ({0+\shift},2.5);

  \draw[
  red, thick, decoration={markings, mark=at position 0.6 with {\arrow[red,thick]{Stealth}}},
  postaction={decorate}
  ]
  ({-2.121-0.1+\shift},2.121) to[out=-45, in=135, looseness=1] ({-1.414-0.1+\shift},1.414) to[out=-45, in=180, looseness=1] ({-1+\shift},0) to[out=0, in=180, looseness=1] ({0.25+\shift},0) to[out=0, in=-90, looseness=1] ({0.35+\shift},0.9) to[out=90,in=-45,looseness=1] ({-1.414+0.1+\shift},1.414) to[out=135, in=-45, looseness=1] ({-2.121+\shift},{2.121+0.1});

  \fill ({0+\shift},0) circle (2.5pt);
  \node[below right] at ({-0.1+\shift},0.1) {$0$};

  \fill ({-1+\shift},0.6) circle (2.5pt);
  \node[right] at ({-1+\shift},0.6) {$p_1$};

  \fill ({0+\shift},0.6) circle (2.5pt);
  \node[above] at ({0+\shift},0.6) {$p_m$};

  \fill ({0.7+\shift},0.6) circle (2.5pt);
  \node[above] at ({0.7+\shift},0.6) {$p_N$};

  \fill ({-1+\shift},-0.6) circle (2.5pt);
  \node[right] at ({-1+\shift},-0.6) {$p_1^*$};

  \fill ({0+\shift},-0.6) circle (2.5pt);
  \node[below] at ({0+\shift},-0.6) {$p_m^*$};

  \fill ({0.7+\shift},-0.6) circle (2.5pt);
  \node[below] at ({0.7+\shift},-0.6) {$p_N^*$};

\node[above, left] at (-2.121+1.5+\shift,2.121+0.4) {$\infty\ee^{3\pi\ii/4}$};

  \draw[gray] node at ({2+\shift},2) {$\boxed{z}$};

\def\shift{13}

  \draw[gray, thin, ->] ({-2.5+\shift},0) -- ({2.5+\shift},0);
  \draw[gray, thin, ->] ({0+\shift},-2.5) -- ({0+\shift},2.5);
  \draw[
  red, thick, decoration={markings, mark=at position 0.7 with {\arrow[red,thick]{Stealth}}},
  postaction={decorate}
  ]
  ({-2.121+\shift},2.121) to[out=-45, in=135, looseness=1] ({-1.414+\shift},1.414) to[out=-45, in=180, looseness=1] ({-1+\shift},0) to[out=0, in=180, looseness=1] ({0.25+\shift},0) to[out=0, in=135, looseness=1] ({0+\shift},0.6);

  \fill ({0+\shift},0) circle (2.5pt);
  \node[below right] at ({-0.1+\shift},0.1) {$0$};

  \fill ({-1+\shift},0.6) circle (2.5pt);
  \node[right] at ({-1+\shift},0.6) {$p_1$};

  \fill ({0+\shift},0.6) circle (2.5pt);
  \node[above] at ({0+\shift},0.6) {$p_m$};

  \fill ({0.7+\shift},0.6) circle (2.5pt);
  \node[above] at ({0.7+\shift},0.6) {$p_N$};

  \fill ({-1+\shift},-0.6) circle (2.5pt);
  \node[right] at ({-1+\shift},-0.6) {$p_1^*$};

  \fill ({0+\shift},-0.6) circle (2.5pt);
  \node[below] at ({0+\shift},-0.6) {$p_m^*$};

  \fill ({0.7+\shift},-0.6) circle (2.5pt);
  \node[below] at ({0.7+\shift},-0.6) {$p_N^*$};

\node[above, left] at (-2.121+1.5+\shift,2.121+0.4) {$\infty\ee^{3\pi\ii/4}$};

  \draw[gray] node at ({2+\shift},2) {$\boxed{z}$};

\end{tikzpicture}    
\end{center}
\caption{Left:  the contour $C_0$.  The contour $C_m$ for non-exceptional $m$ (center) and for exceptional $m$ (right).}
\label{fig:C0Cm}
\end{figure}
Recalling the antiderivative $L(z)$ of $u_0(z)$ that is single valued for large $|z|$ with $L(\infty)=0$, we define $L(z)$ along each contour $C_0,\dots,C_N$ by analytic continuation from this initial value.  Hence if $n=1,\dots,N$ is a non-exceptional index, then as $z\to\infty$ in the direction of orientation along $C_n$, one has that $\ee^{-\ii L(z)}\to E_n$ where $E_n$ is defined in \eqref{eq:XS-En-non-exceptional}.  Similarly, 
\begin{equation}
\mathop{\lim_{z\to \infty\ee^{-\ii\pi/4}}}_{z\in C_0}\ee^{-\ii L(z)} = E_0:= \ee^{2\pi(c_1+\cdots +c_N)}.
\label{eq:E0}
\end{equation}
We also recall our labeling convention that if there exists any non-exceptional index, then $N$ is non-exceptional.
We may then define a function $h(z)=h(z;t,x)$ on each contour $C_0,\dots,C_N$ by
\begin{equation}
    h(z;t,x):=\frac{(z-x)^2}{4t}+L(z),\quad x\in\mathbb{R},\quad t>0,\quad z\in \bigcup_{n=0}^NC_n.
\end{equation}

With these ingredients, we now give the solution of the Cauchy problem for \eqref{eq:XS-BO} with the initial datum \eqref{eq:XS-u0-rational}, by first defining two $(N+1)\times (N+1)$ matrices.  We set
\begin{equation}
\begin{aligned}
    A_{j1}(t,x)&:=\int_{C_{j-1}}u_0(z)\ee^{-\ii h(z;t,x)}\,\dd z\\
    B_{j1}(t,x)&:=\int_{C_{j-1}}\ee^{-\ii h(z;t,z)}\,\dd z
\end{aligned}
\label{eq:XS-AB-first-column}
\end{equation}
for $j=1,\dots,N+1$, and then 
\begin{equation}
    A_{jk}(t,x)=B_{jk}(t,x):=\int_{C_{j-1}}\frac{\ee^{-\ii h(z;t,x)}}{z-p_{k-1}}\,\dd z,\quad j=1,\dots,N+1,\; k=2,\dots,N+1.
\label{eq:XS-AB-other-elements}
\end{equation}
Then we have the following adaptation to rational initial data of the explicit formula first obtained in \cite{Gerard22}.
\begin{theorem}[\protect{Solution of the Benjamin-Ono Cauchy problem, \cite[Theorem 4]{BlackstoneGGM24a}}]
The solution of the Benjamin-Ono equation \eqref{eq:XS-BO} with initial condition \eqref{eq:XS-u0-rational} is given by $u(t,x)=\Pi u(t,x)+\Pi u(t,x)^*$, where
\begin{equation}
    \Pi u(t,x)=\frac{\det(\mathbf{A}(t,x))}{\det(\mathbf{B}(t,x))},\quad x\in\mathbb{R},\quad t>0.
\end{equation}
\label{thm:XS-exact}
\end{theorem}

\subsection{Expansion of matrix elements}
Fixing a bounded interval containing the similarity variable $\xi\in\mathbb{R}$ and writing $x=2 t^{1/2}\xi$, we now find the asymptotic behavior of the elements of $\mathbf{A}(t,x)$ and $\mathbf{B}(t,x)$ as $t\to+\infty$.
\subsubsection{First column of $\mathbf{A}(t,x)$}
According to \eqref{eq:XS-AB-first-column}, the integrals in the first column of the matrix $\mathbf{A}(t,2\xi\sqrt{t})$ are 
\begin{equation}
    A_{j1}(t,2\xi\sqrt{t})=\int_{C_{j-1}}u_0(z)\ee^{-\ii h(z;t,2\xi\sqrt{t})}\,\dd z=\ee^{-\ii\xi^2}\int_{C_{j-1}}u_0(z)\ee^{-\ii L(z)}\ee^{-\ii z^2/(4t)}\ee^{\ii \xi z/\sqrt{t}}\,\dd z.
\end{equation}
Since $u_0\ee^{-\ii L}\in L^1(C_{j-1})$ and is independent of $t$ and for each $\xi\in\mathbb{R}$, $z\mapsto\ee^{-\ii z^2/(4t)}\ee^{\ii\xi z/\sqrt{t}}$ has an upper bound on $C_{j-1}$ that is independent of $t$, it follows by dominated convergence that
\begin{equation}
    A_{j1}(t,2\xi\sqrt{t})=\ee^{-\ii\xi^2}\int_{C_{j-1}}u_0(z)\ee^{-\ii L(z)}\,\dd z + o(1),\quad t\to+\infty.
\label{eq:A-first-column-simple-estimate}
\end{equation}
We can improve the estimate by using the fact that on $C_{j-1}$ we have both $\ee^{-\ii z^2/(4t)}\ee^{\ii\xi z/\sqrt{t}}-1=O(1)$ and $\ee^{-\ii z^2/(4t)}\ee^{\ii\xi z/\sqrt{t}}-1=O(z/\sqrt{t})$, both holding uniformly for bounded $\xi$.  Since also $u_0(z)\ee^{-\ii L(z)}=O(1/(1+|z|^2))$ on $C_{j-1}$,
\begin{multline}
    \left|\int_{C_{j-1}}u_0(z)\ee^{-\ii L(z)}\left[\ee^{-\ii z^2/(4t)}\ee^{\ii\xi z/\sqrt{t}}-1\right]\,\dd z\right|\\
\begin{aligned}
    &\lesssim \int_{C_{j-1}}\frac{\min\{1,|z|/\sqrt{t}\}}{1+|z|^2}\,|\dd z|\\ &=\frac{1}{\sqrt{t}}\int_{C_{j-1}\cap\{|z|<\sqrt{t}\}}\frac{|z|\,|\dd z|}{1+|z|^2} +\int_{C_{j-1}\cap\{|z|\ge\sqrt{t}\}}\frac{|\dd z|}{1+|z|^2}.
\end{aligned}
\end{multline}
Clearly, the second term is $O(1/\sqrt{t})$ while the first term is $O(\ln(t)/\sqrt{t})$.  Therefore the error term $o(1)$ in \eqref{eq:A-first-column-simple-estimate} can be replaced with the sharper $O(\ln(t)/\sqrt{t})$ as $t\to+\infty$, which is uniform for bounded $\xi$.

Finally, we use $u_0(z)=L'(z)$ to explicitly evaluate the leading term.  If $j=1$, or if $j=2,\dots,N+1$ and $j-1$ is a non-exceptional index, then  
\begin{equation}
\begin{aligned}
    A_{j1}(t,2\xi\sqrt{t})&=\ii\ee^{-\ii\xi^2}\left[\ee^{-\ii L(z)}\right]_{\partial C_{j-1}} + O\left(\frac{\ln(t)}{\sqrt{t}}\right)\\ &=\ii\ee^{-\ii\xi^2}\left(E_{j-1}-1\right)+O\left(\frac{\ln(t)}{\sqrt{t}}\right),\quad t\to+\infty,\quad j=1,\dots,N+1.
\end{aligned}
    \label{eq:A-first-column-large-t}
\end{equation}
On the other hand, if $j>1$ and $j-1$ is an exceptional index, then $C_{j-1}$ terminates at $p_{j-1}$ and $\ee^{-\ii L(z)}$ vanishes to at least first order at $z=p_{j-1}$, so exactly the same result holds taking into account the definition \eqref{eq:XS-En-exceptional} of $E_{j-1}$ in the exceptional case.

\subsubsection{First column of $\mathbf{B}(t,x)$}
For $n=1,\dots,N$, recall the splitting of $C_n$ into an initial arc $C_{n}^{-}$ from $z=\infty\ee^{3\pi\ii/4}$ to $z=0$ and a terminal arc $C_{n}^{+}$ starting from $z=0$.  We similarly assume $C_0$ passes through the origin and split $C_0$ into $C_{0}^{-}$ and $C_{0}^{+}$.  According to \eqref{eq:XS-AB-first-column}, the elements of the first column of the matrix $\mathbf{B}(t,2\xi\sqrt{t})$ can then be written as
\begin{equation}
    B_{j1}(t,2\xi\sqrt{t})=\int_{C_{j-1}^{-}}\ee^{-\ii h(z;t,2\xi\sqrt{t})}\,\dd z + \int_{C_{j-1}^{+}}\ee^{-\ii h(z;t,2\xi\sqrt{t})}\,\dd z ,\quad j=1,\dots,N+1.
\end{equation}
Using the property that $\ee^{-\ii L(z)}\to 1$ as $z\to\infty$ on $C_{j-1}^{-}$ for all $j=1,\dots,N+1$,
\begin{equation}
\begin{aligned}
    \int_{C_{j-1}^{-}}\ee^{-\ii h(z;t,2\xi\sqrt{t})}\,\dd z &= \int_{C_{j-1}^{-}}\ee^{-\ii (z-2\xi\sqrt{t})^2/(4t)}\,\dd z + \int_{C_{j-1}^{-}}\left[\ee^{-\ii L(z)}-1\right]\ee^{-\ii (z-2\xi\sqrt{t})^2/(4t)}\,\dd z\\
    &=2\sqrt{t}\int_{\infty\ee^{3\pi\ii/4}}^0\ee^{-\ii (\zeta-\xi)^2}\,\dd\zeta + \int_{C_{j-1}^{-}}\left[\ee^{-\ii L(z)}-1\right]\ee^{-\ii (z-2\xi\sqrt{t})^2/(4t)}\,\dd z,
\end{aligned}
\end{equation}
where we rescaled by $z=2\sqrt{t}\zeta$ in the first integral.  We estimate the second integral by using that $|\ee^{-\ii L(z)}-1|\lesssim\min\{1,|z|^{-1}\}$ and that $\ee^{-\ii (z-2\xi\sqrt{t})^2/(4t)}=O(1)$ on $C_{j-1}^{-}$.  Therefore
\begin{multline}
        \left|\int_{C_{j-1}^{-}}\left[\ee^{-\ii L(z)}-1\right]\ee^{-\ii (z-2\xi\sqrt{t})^2/(4t)}\,\dd z\right|\\
        \lesssim \int_{C_{j-1}^{-}\cap\{|z|\le 1\}}\,|\dd z| + \int_{C_{j-1}^{-}\cap\{|z|>1\}}\left|\ee^{-\ii (z-2\xi\sqrt{t})^2/(4t)}\right|\frac{|\dd z|}{|z|}.
\end{multline}
The first term is independent of $t$, and by rescaling by $z=2\sqrt{t}\zeta$, the second integral is $O(\ln(t))$ as $t\to+\infty$.  We hence obtain
\begin{equation}
\int_{C_{j-1}^{-}}\ee^{-\ii h(z;t,2\xi\sqrt{t})}\,\dd z = 2\sqrt{t}\int_{\infty\ee^{3\pi\ii/4}}^0\ee^{-\ii (\zeta-\xi)^2}\,\dd\zeta + O(\ln(t))\quad t\to+\infty,
    \label{eq:B-first-column-C-}
\end{equation}
with the estimate being uniform for bounded $\xi$.

Similarly, if $j=1$ or if $j>1$ and $j-1$ is a non-exceptional index, then $\ee^{-\ii L(z)}\to E_{j-1}$ as $z\to\infty$ along $C_{j-1}^{+}$, so
\begin{multline}
\int_{C_{j-1}^{+}}\ee^{-\ii h(z;t,2\xi\sqrt{t})}\,\dd z \\    \begin{aligned}
&= E_{j-1}\int_{C_{j-1}^{+}}\ee^{-\ii (z-2\xi\sqrt{t})^2/(4t)}\,\dd z + \int_{C_{j-1}^{+}}\left[\ee^{-\ii L(z)}-E_{j-1}\right]\ee^{-\ii (z-2\xi\sqrt{t})^2/(4t)}\,\dd z\\ 
&=2E_{j-1}\sqrt{t}\int_0^{\infty\ee^{\ii\theta_{j-1}}}\ee^{-\ii (\zeta-\xi)^2}\,\dd\zeta + O(\ln(t)),\quad t\to+\infty,
    \end{aligned}
\label{eq:B-first-column-C+}
\end{multline}
where $\theta_n$ is the angle with which $C_{n}^{+}$ tends to infinity, namely $\theta_0=-\pi/4$ and $\theta_n=3\pi/4$ for $n=1,\dots,N$.
Combining \eqref{eq:B-first-column-C+} with \eqref{eq:B-first-column-C-}, the following holds uniformly for bounded $\xi$:
\begin{multline}
    B_{j1}(t,2\xi\sqrt{t})=2\sqrt{t}\left[\int_{\infty\ee^{3\pi\ii/4}}^0\ee^{-\ii (\zeta-\xi)^2}\,\dd\zeta + E_{j-1}\int_0^{\infty\ee^{\ii\theta_{j-1}}}\ee^{-\ii (\zeta-\xi)^2}\,\dd\zeta\right] + O(\ln(t)),\\ t\to+\infty,\quad j=1,\dots,N+1.
    \label{eq:B-first-column-large-t}
\end{multline}
If $j-1$ is an exceptional index (necessarily with $j\ge 2$), then $C_{j-1}^{+}$ is a bounded arc from $z=0$ to $z=p_{j-1}$, and the integral on the left-hand side in \eqref{eq:B-first-column-C+} remains bounded as $t\to+\infty$. Hence we again obtain \eqref{eq:B-first-column-large-t} because $E_{j-1}=0$ according to \eqref{eq:XS-En-exceptional}.

\subsubsection{Remaining elements of $\mathbf{A}(t,x)$ and $\mathbf{B}(t,x)$}
According to \eqref{eq:XS-AB-other-elements}, the elements of the remaining $N$ columns of $\mathbf{A}(t,x)$ and $\mathbf{B}(t,x)$ agree and are given for $j=1,\dots,N+1$ and $k=2,\dots,N+1$ by
\begin{equation}
    A_{jk}(t,2\xi\sqrt{t})=B_{jk}(t,2\xi\sqrt{t})=\int_{C_{j-1}}\frac{\ee^{-\ii h(z;t,2\xi\sqrt{t})}\,\dd z}{z-p_{k-1}}=\ee^{-\ii\xi^2}\int_{C_{j-1}}\frac{\ee^{-\ii L(z)}\ee^{-\ii z^2/(4t)}\ee^{\ii \xi z/\sqrt{t}}\,\dd z}{z-p_{k-1}}.
    \label{eq:AjkBjk}
\end{equation}

We again use the splitting at $z=0$ of $C_{j-1}$ into $C_{j-1}^{-}$ and $C_{j-1}^{+}$.  First we analyze the contribution from $C_{j-1}^{-}$:
\begin{multline}
    \ee^{-\ii\xi^2}\int_{C_{j-1}^{-}}\frac{\ee^{-\ii L(z)}\ee^{-\ii z^2/(4t)}\ee^{\ii \xi z/\sqrt{t}}\,\dd z}{z-p_{k-1}}\\
    \begin{aligned} 
    &=\ee^{-\ii\xi^2}\int_{C_{j-1}^{-}}\frac{\ee^{-\ii z^2/(4t)}\ee^{\ii\xi z/\sqrt{t}}\,\dd z}{z-p_{k-1}} +\ee^{-\ii\xi^2}\int_{C_{j-1}^{-}}\frac{[\ee^{-\ii L(z)}-1]\ee^{-\ii z^2/(4t)}\ee^{\ii\xi z/\sqrt{t}}\,\dd z}{z-p_{k-1}}\\
    &=\int_{\widetilde{C}_{j-1}^{-}}\frac{\ee^{-\ii (\zeta-\xi)^2}\,\dd\zeta}{\zeta-p_{k-1}/(2\sqrt{t})} +\ee^{-\ii\xi^2}\int_{C_{j-1}^{-}}\frac{\ee^{-\ii L(z)}-1}{z-p_{k-1}}\,\dd z + o(1)
    \end{aligned}
\end{multline}
as $t\to+\infty$ because $z\mapsto (\ee^{-\ii L(z)}-1)/(z-p_{k-1})$ is in $L^1(C_{j-1}^{-})$ and is independent of $t$.  In the first term, we rescaled by $z=2\sqrt{t}\zeta$, and $\widetilde{C}_{j-1}^{-}$ refers to the corresponding rescaled contour in the $\zeta$-plane.  With the substitution $w=\zeta-p_{k-1}/(2\sqrt{t})$, we get
\begin{equation}
    \int_{\widetilde{C}_{j-1}^{-}}\frac{\ee^{-\ii (\zeta-\xi)^2}\,\dd\zeta}{\zeta-p_{k-1}/(2\sqrt{t})} = \int_{\widetilde{C}_{j-1}^{-\prime}}\ee^{-\ii (w+p_{k-1}/(2\sqrt{t})-\xi)^2}\frac{\dd w}{w},
\end{equation}
where $\widetilde{C}_{j-1}^{-\prime}=\widetilde{C}_{j-1}^--p_{k-1}/(2\sqrt{t})=(C_{j-1}^--p_{k-1})/(2\sqrt{t})$ is a contour originating at $w=\infty\ee^{3\pi\ii/4}$ and terminating at $w=-p_{k-1}/(2\sqrt{t})$. Note that this terminal point lies near the origin in the lower half $w$-plane.  By a contour deformation avoiding $w=0$, we therefore obtain
\begin{equation}
\begin{split}
       \int_{\widetilde{C}_{j-1}^{-}}\frac{\ee^{-\ii (\zeta-\xi)^2}\,\dd\zeta}{\zeta-p_{k-1}/(2\sqrt{t})} &= \int_{W^-}\ee^{-\ii (w+p_{k-1}/(2\sqrt{t})-\xi)^2}\frac{\dd w}{w} \\ &\qquad{}+\int_1^{-p_{k-1}/(2\sqrt{t})}\frac{\ee^{-\ii(w+p_{k-1}/(2\sqrt{t})-\xi)^2}-\ee^{-\ii (p_{k-1}/(2\sqrt{t})-\xi)^2}}{w}\,\dd w\\
       &\qquad{}+\ee^{-\ii (p_{k-1}/(2\sqrt{t})-\xi)^2}\int_1^{-p_{k-1}/(2\sqrt{t})}\frac{\dd w}{w}\\
       &=-\frac{\ee^{-\ii\xi^2}}{2}\ln(t) +\int_{W^-}\ee^{-\ii (w-\xi)^2}\frac{\dd w}{w} \\&\qquad{}+ \int_1^0\frac{\ee^{-\ii (w-\xi)^2}-\ee^{-\ii\xi^2}}{w}\,\dd w + \ee^{-\ii\xi^2}\log\left(\frac{-p_{k-1}}{2}\right)+o(1).
\end{split}
\end{equation}
Here, the path of integration $W^-$ in the first integral is the concatenation of $\widetilde{C}_{j-1}^{-\prime}$ followed by a path in the lower half-plane from $w=-p_{k-1}/(2\sqrt{t})$ to $w=1$ (and hence passing \emph{below} the pole at the origin, so Cauchy-equivalent to the straight ray from $w=\infty\ee^{3\pi\ii/4}$ to $w=\ee^{3\pi\ii/4}$ followed by the arc of the unit circle with $3\pi/4<\arg(w)<2\pi$ as shown in Figure~\ref{fig:W-}), and the logarithm denotes the principal branch.
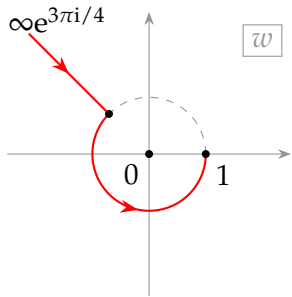
\begin{figure}[h]
\begin{center}
\begin{tikzpicture}[>=Stealth,scale=0.75]

  \draw[gray, thin, ->] (-2.5,0) -- (2.5,0);
  \draw[gray, thin, ->] (0,-2.5) -- (0,2.5);
  \draw[gray, thin, dashed] (0,0) circle (1);

  \draw[
    red, thick,
    decoration={
      markings,
      mark=at position 0.5 with {\arrow[red, thick]{Stealth}}
    },
    postaction={decorate}
  ]
  ({3*cos(135)}, {3*sin(135)}) -- ({cos(135)}, {sin(135)});

  \draw[
    red, thick,
    decoration={
      markings,
      mark=at position 0.55 with {\arrow[red, thick]{Stealth}}
    },
    postaction={decorate}
  ]
  ({cos(135)}, {sin(135)}) arc[start angle=135, end angle=360, radius=1];

  \node[below left] at (0,0) {$0$};

  \fill (1,0) circle (2pt);
  \node[below right] at (1,0) {$1$};

  \fill ({cos(135)},{sin(135)}) circle (2pt);

    \node at ({3*cos(135)+0.5},{3*sin(135)+0.25}) {$\infty\ee^{3\pi\ii/4}$};

    \draw[gray] node at (2,2) {$\boxed{w}$};

  \fill (0,0) circle (2pt);

\end{tikzpicture}
    
\end{center}
\caption{The contour $W^-$.}
\label{fig:W-}
\end{figure}
Summarizing, we have
\begin{multline}
    \ee^{-\ii\xi^2}\int_{C_{j-1}^{-}}\frac{\ee^{-\ii L(z)}\ee^{-\ii z^2/(4t)}\ee^{\ii \xi z/\sqrt{t}}\,\dd z}{z-p_{k-1}}\\
    \begin{aligned} 
    &=-\frac{\ee^{-\ii\xi^2}}{2}\ln(t)+\int_{W^-}\ee^{-\ii (w-\xi)^2}\frac{\dd w}{w} \\&\qquad{}+ \int_1^0\frac{\ee^{-\ii (w-\xi)^2}-\ee^{-\ii\xi^2}}{w}\,\dd w + \ee^{-\ii\xi^2}\log\left(\frac{-p_{k-1}}{2}\right)\\&\qquad{}+\ee^{-\ii\xi^2}\int_{C_{j-1}^{-}}\frac{\ee^{-\ii L(z)}-1}{z-p_{k-1}}\,\dd z+o(1).
    \end{aligned}    
    \label{eq:AjkBjk-minus}
\end{multline}

We can similarly analyze the contribution from $C_{j-1}^{+}$, assuming that either $j=1$ or $j>1$ and $j-1$ is a non-exceptional index.  In either case $\ee^{-\ii L(z)}\to E_{j-1}$ as $z\to\infty$ along $C_{j-1}^{+}$:
\begin{multline}
    \ee^{-\ii\xi^2}\int_{C_{j-1}^{+}}\frac{\ee^{-\ii L(z)}\ee^{-\ii z^2/(4t)}\ee^{\ii \xi z/\sqrt{t}}\,\dd z}{z-p_{k-1}}\\
    \begin{aligned} 
    &=E_{j-1}\ee^{-\ii\xi^2}\int_{C_{j-1}^{+}}\frac{\ee^{-\ii z^2/(4t)}\ee^{\ii\xi z/\sqrt{t}}\,\dd z}{z-p_{k-1}} +\ee^{-\ii\xi^2}\int_{C_{j-1}^{+}}\frac{[\ee^{-\ii L(z)}-E_{j-1}]\ee^{-\ii z^2/(4t)}\ee^{\ii\xi z/\sqrt{t}}\,\dd z}{z-p_{k-1}}\\
    &=E_{j-1}\int_{\widetilde{C}_{j-1}^{+}}\frac{\ee^{-\ii (\zeta-\xi)^2}\,\dd\zeta}{\zeta-p_{k-1}/(2\sqrt{t})} +\ee^{-\ii\xi^2}\int_{C_{j-1}^{+}}\frac{\ee^{-\ii L(z)}-E_{j-1}}{z-p_{k-1}}\,\dd z + o(1),
    \end{aligned}
\label{eq:Cjminus1plus-split}
\end{multline}
where $\widetilde{C}_{j-1}^{+}$ is the rescaled contour in the $\zeta=z/(2\sqrt{t})$ plane, and 
\begin{multline}
    \int_{\widetilde{C}_{j-1}^{+}}\frac{\ee^{-\ii (\zeta-\xi)^2}\,\dd\zeta}{\zeta-p_{k-1}/(2\sqrt{t})} = \frac{\ee^{-\ii\xi^2}}{2}\ln(t) +\int_{W^+_{jk}}\ee^{-\ii(w-\xi)^2}\frac{\dd w}{w}\\+\int_0^1\frac{\ee^{-\ii (w-\xi)^2}-\ee^{-\ii\xi^2}}{w}\,\dd w -\ee^{-\ii\xi^2}\log\left(\frac{-p_{k-1}}{2}\right)+o(1).
\end{multline}
Here, the logarithm denotes the principal branch, and $W^+_{jk}$ denotes the concatenation of a path in $\mathbb{C}_-$ from $w=1$ to $w=-p_{k-1}/(2\sqrt{t})$ followed by $\widetilde{C}_{j-1}^{+\prime}=(C_{j-1}^{+}-p_{k-1})/(2\sqrt{t})$.  By Cauchy's theorem, we may take $W^+_{jk}$ to be the contour shown in Figure~\ref{fig:W+}.  
\begin{figure}[h]
\begin{center}
    \begin{tikzpicture}[>=Stealth, scale=0.75]
  \draw[gray, thin, ->] (-2.5,0) -- (2.5,0);
  \draw[gray, thin, ->] (0,-2.5) -- (0,2.5);
  \draw[gray, thin, dashed] (0,0) circle (1);

  \draw[
    red, thick,
    decoration={
      markings,
      mark=at position 0.7 with {\arrow[red, thick]{Stealth}}
    },
    postaction={decorate}
  ]
  (1,0) arc[start angle=0, end angle=-45, radius=1];

  \draw[
    red, thick,
    decoration={
      markings,
      mark=at position 0.5 with {\arrow[red, thick]{Stealth}}
    },
    postaction={decorate}
  ]
  ({cos(-45)}, {sin(-45)}) -- ({3*cos(-45)}, {3*sin(-45)});

  \fill (0,0) circle (2.5pt);
  \node[below left] at (0,0) {$0$};

  \fill (1,0) circle (2.5pt);
  \node[above right] at (1,0) {$1$};

  \node at ({3*cos(-45)+1},{3*sin(-45)}) {$\infty\ee^{-\ii\pi/4}$};

  \draw[gray] node at (2,2) {$\boxed{w}$};

  \fill ({cos(-45)},{sin(-45)}) circle (2.5pt);
  
    \def\shift{6.5}

  \draw[gray, thin, ->] (-2.5+\shift,0) -- (2.5+\shift,0);
  \draw[gray, thin, ->] (0+\shift,-2.5) -- (0+\shift,2.5);
  \draw[gray, thin, dashed] (\shift,0) circle (1);

  \draw[
    red, thick,
    decoration={
      markings,
      mark=at position 0.5 with {\arrowreversed[red, thick]{Stealth}}
    },
    postaction={decorate}
  ]
  ({3*cos(135)+\shift}, {3*sin(135)}) -- ({cos(135)+\shift}, {sin(135)});

   \draw[
    red, thick,
    decoration={
      markings,
      mark=at position 0.55 with {\arrowreversed[red, thick]{Stealth}}
    },
    postaction={decorate}
  ]
  ({cos(135)+\shift}, {sin(135)}) arc[start angle=135, end angle=360, radius=1];

  \fill (\shift,0) circle (2.5pt);
  \node[below left] at (\shift,0) {$0$};

  \fill (1+\shift,0) circle (2.5pt);
  \node[above right] at (1+\shift,0) {$1$};

  \fill ({cos(135)+\shift},{sin(135)}) circle (2.5pt);

    \node at ({3*cos(135)+\shift+0.5},{3*sin(135)+0.25}) {$\infty\ee^{3\pi\ii/4}$};

    \draw[gray] node at ({2+\shift},2) {$\boxed{w}$};

\def\shift{13}

  \draw[gray, thin, ->] (-2.5+\shift,0) -- (2.5+\shift,0);
  \draw[gray, thin, ->] (0+\shift,-2.5) -- (0+\shift,2.5);
  \draw[gray, thin, dashed] (\shift,0) circle (1);

  \draw[
    red, thick,
    decoration={
      markings,
      mark=at position 0.5 with {\arrowreversed[red, thick]{Stealth}}
    },
    postaction={decorate}
  ]
  ({3*cos(135)+\shift}, {3*sin(135)}) -- ({cos(135)+\shift}, {sin(135)});

  \draw[
    red, thick,
    decoration={
      markings,
      mark=at position 0.55 with {\arrowreversed[red, thick]{Stealth}}
    },
    postaction={decorate}
  ]
  ({cos(135)+\shift}, {sin(135)}) arc[start angle=135, end angle=0, radius=1];

  \fill (\shift,0) circle (2.5pt);
  \node[below left] at (\shift,0) {$0$};

  \fill (1+\shift,0) circle (2.5pt);
  \node[below right] at (1+\shift,0) {$1$};

  \fill ({cos(135)+\shift},{sin(135)}) circle (2.5pt);

       \node at ({3*cos(135)+\shift+0.5},{3*sin(135)+0.25}) {$\infty\ee^{3\pi\ii/4}$};

    \draw[gray] node at ({2+\shift},2) {$\boxed{w}$};
\end{tikzpicture}
\end{center}
\caption{The contour $W^+_{jk}$ for $j=1$ (left), and for non-exceptional $j-1$ with $1<j<k$ (center) and $j\ge k>1$ (right).}
\label{fig:W+}
\end{figure}
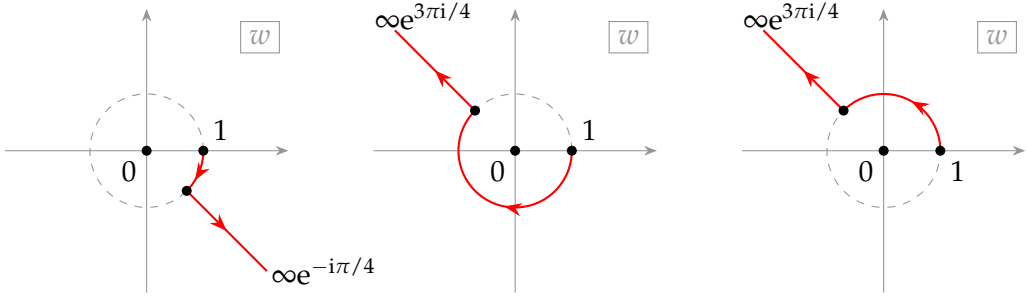

Therefore,
\begin{multline}
    \ee^{-\ii\xi^2}\int_{C_{j-1}^{+}}\frac{\ee^{-\ii L(z)}\ee^{-\ii z^2/(4t)}\ee^{\ii\xi z/\sqrt{t}}\,\dd z}{z-p_{k-1}}\\
    \begin{aligned}
        &=\frac{E_{j-1}\ee^{-\ii\xi^2}}{2}\ln(t)+E_{j-1}\int_{W^+_{jk}}\ee^{-\ii(w-\xi)^2}\frac{\dd w}{w}\\
        &\qquad{}+E_{j-1}\int_0^1\frac{\ee^{-\ii (w-\xi)^2}-\ee^{-\ii\xi^2}}{w}\,\dd w-E_{j-1}\ee^{-\ii\xi^2}\log\left(\frac{-p_{k-1}}{2}\right)\\
        &\qquad{}+\ee^{-\ii\xi^2}\int_{C_{j-1}^{+}}\frac{\ee^{-\ii L(z)}-E_{j-1}}{z-p_{k-1}}\,\dd z +o(1).
    \end{aligned}
    \label{eq:AjkBjk-plus}
\end{multline}
On the other hand, if $j-1$ is an exceptional index, then $C_{j-1}^{+}$ terminates at $p_{j-1}$, and there is no need to split up the integral in \eqref{eq:Cjminus1plus-split}.  Applying dominated convergence directly, one again obtains \eqref{eq:AjkBjk-plus} because $E_{j-1}=0$ according to \eqref{eq:XS-En-exceptional}.

Combining the contributions \eqref{eq:AjkBjk-minus} from $C_{j-1}^{-}$ and \eqref{eq:AjkBjk-plus} from $C_{j-1}^{+}$ in \eqref{eq:AjkBjk} gives that for all $j=1,\dots,N+1$ and $k=2,\dots,N+1$,
\begin{multline}
    A_{jk}(t,2\xi\sqrt{t})=B_{jk}(t,2\xi\sqrt{t}) = (E_{j-1}-1)\frac{\ee^{-\ii\xi^2}}{2}\ln(t)
    \\+E_{j-1}\int_{W^+_{jk}}\ee^{-\ii (w-\xi)^2}\frac{\dd w}{w} +\int_{W^-}\ee^{-\ii(w-\xi)^2}\frac{\dd w}{w}
    \\
    +(E_{j-1}-1)\int_0^1\frac{\ee^{-\ii(w-\xi)^2}-\ee^{-\ii\xi^2}}{w}\,\dd w+
    K_{jk}\ee^{-\ii\xi^2} + o(1)
    \label{eq:AjkBjk-large-t}
\end{multline}
holds as $t\to+\infty$ uniformly for bounded $\xi$, where $K_{jk}$ are constants defined by
\begin{equation}
    K_{jk}:=    (1-E_{j-1})\log\left(\frac{-p_{k-1}}{2}\right)
    +\int_{C_{j-1}^{-}}\frac{\ee^{-\ii L(z)}-1}{z-p_{k-1}}\,\dd z +\int_{C_{j-1}^{+}}\frac{\ee^{-\ii L(z)}-E_{j-1}}{z-p_{k-1}}\,\dd z. 
    \label{eq:Kjk}
\end{equation}  
\begin{lemma}
If there exists a non-exceptional index (and so by convention $N$ is non-exceptional), then 
    $K_{1k}=K_{N+1,k}$ holds for $k=2,\dots,N+1$.
    Otherwise (if all indices are exceptional), then $K_{1k}=-2\pi\ii$ for all $k=2,\dots,N+1$.
    \label{lem:firstrow-lastrow}
\end{lemma}
\begin{proof}
Suppose first that $N$ is non-exceptional.  
Since for all $j=1,\dots,N+1$, $C_{j-1}^-$ always originates at $z=\infty\ee^{3\pi\ii/4}$ and terminates at $z=0$, the integrals in \eqref{eq:Kjk} over $C_{0}^{-}$ and $C_{N}^{-}$ agree for all $k$.  
Furthermore, comparing \eqref{eq:XS-En-non-exceptional} with  \eqref{eq:E0} gives $E_0=E_{N}$.  Therefore the logarithmic terms in $K_{1k}$ and $K_{N+1,k}$ agree, as do the integrands over $C_{0}^{+}$ and $C_{N}^{+}$.  Moreover, both of these common integrands are analytic and $O(z^{-2})$ for $z$ on and in between the contours $C_0^+$ and $C_N^+$, so the integrals over $C_{0}^{+}$ and $C_{N}^{+}$ also agree, proving that $K_{1k}=K_{N+1,k}$ for $k=2,\dots,N+1$.

On the other hand, suppose that all indices are exceptional.  Then by \eqref{eq:E0} we have $E_0=1$ because all residues $c_1,\dots,c_N$ are imaginary integers.  Therefore, the logarithmic term in $K_{1k}$ vanishes, and also the two integrals have the same integrand so they can be combined into a single integral over $C_0$.  In that resulting integral, the path of integration can be closed in the upper half-plane and the integral evaluated by residues.  Since $\ee^{-\ii L(z)}$ vanishes at $z=p_n$ for all $n=1,\dots,N$, the only contribution comes from the residue of $-1/(z-p_{k-1})$ at $z=p_{k-1}$, proving that $K_{1k}=-2\pi\ii$ for $k=2,\dots,N+1$.
\end{proof}

\subsection{Computing $\det(\mathbf{A}(t,2\xi\sqrt{t}))$}
Setting $\boldsymbol{\eta}:=(E_0-1,\dots,E_{N}-1)^\top$,
we express each column $k=1,\dots,N+1$ of $\mathbf{A}(t,2\xi\sqrt{t})$ in the form $c_k\boldsymbol{\eta}+\boldsymbol{\gamma}^{(k)}$ for a scalar $c_k$ and a vector $\boldsymbol{\gamma}^{(k)}$.  Thus, $\mathbf{A}(t,2\xi\sqrt{t})$ is a rank-one perturbation of a matrix $\boldsymbol{\Gamma}:=(\boldsymbol{\gamma}^{(1)},\dots,\boldsymbol{\gamma}^{(N+1)})$.  Indeed, according to \eqref{eq:A-first-column-large-t}, we can choose $c_1=\ii\ee^{-\ii\xi^2}$ and $\boldsymbol{\gamma}^{(1)}=O(\ln(t)/\sqrt{t})$.  Similarly, according to \eqref{eq:AjkBjk-large-t}, we can choose
\begin{equation}
    c_k=\frac{\ee^{-\ii\xi^2}}{2}\ln(t) + \int_0^1\frac{\ee^{-\ii (w-\xi)^2}-\ee^{-\ii\xi^2}}{w}\,\dd w,\quad k=2,\dots,N+1,
\label{eq:ck-remaining-columns}
\end{equation}
and for $j=1,\dots,N+1$ and $k=2,\dots,N+1$,
\begin{equation}
    \Gamma_{jk}=E_{j-1}\int_{W^+_{jk}}\ee^{-\ii (w-\xi)^2}\frac{\dd w}{w} + \int_{W^-}\ee^{-\ii (w-\xi)^2}\frac{\dd w}{w} + K_{jk}\ee^{-\ii\xi^2} + o(1).
\label{eq:Gammajk-remaining-columns}
\end{equation}
Using column multilinearity to expand $\det(\mathbf{A}(t,2\xi\sqrt{t}))$ then gives
\begin{equation}
    \det(\mathbf{A}(t,2\xi\sqrt{t}))=\det(\boldsymbol{\Gamma}) + \sum_{k=1}^{N+1}c_k\det(\boldsymbol{\Gamma}\mathop{\longleftarrow}^k\boldsymbol{\eta}).
\end{equation}
It follows that $\det(\boldsymbol{\Gamma})=O(\ln(t)/\sqrt{t})$, and that if $k>1$, $\displaystyle c_k\det(\boldsymbol{\Gamma}\mathop{\longleftarrow}^k \boldsymbol{\eta})=O(\ln(t)^2/\sqrt{t})$.  Consequently,
\begin{equation}
    \det(\mathbf{A}(t,2\xi\sqrt{t}))=\ii\ee^{-\ii\xi^2}\det(\mathbf{P}(\xi)) + O(t^{-1/2}\ln(t)^2),\quad t\to +\infty,\quad \mathbf{P}(\xi):=\lim_{t\to\infty}\boldsymbol{\Gamma}\mathop{\longleftarrow}^1\boldsymbol{\eta}.
\end{equation}
Note that for $k=2,\dots,N+1$,
\begin{multline}
P_{1k}(\xi):=E_0\int_{W^+_{1k}}\ee^{-\ii (w-\xi)^2}\frac{\dd w}{w} + \int_{W^-}\ee^{-\ii (w-\xi)^2}\frac{\dd w}{w} + K_{1k}\ee^{-\ii\xi^2}
\\
\begin{aligned}
&= (1-E_0)\int_{W^-}\ee^{-\ii (w-\xi)^2}\frac{\dd w}{w} + E_0\int_{L_-}\ee^{-\ii(w-\xi)^2}\frac{\dd w}{w} + K_{1k}\ee^{-\ii\xi^2}\\
&=(1-E_0)\int_{W^-}\ee^{-\ii (w-\xi)^2}\frac{\dd w}{w} + E_0\int_{L_+}\ee^{-\ii(w-\xi)^2}\frac{\dd w}{w} + (2\pi\ii E_0+K_{1k})\ee^{-\ii\xi^2}\\
&=(1-E_0)\int_{W^-}\ee^{-\ii (w-\xi)^2}\frac{\dd w}{w}+E_0\ee^{\ii\pi/4}\sqrt{\pi}F(\xi) + (2\pi\ii E_0+K_{1k})\ee^{-\ii\xi^2}.
\end{aligned}
\label{eq:P-first}
\end{multline}
Here, $L_+$ and $L_-$ are Jordan contours from $w=\infty\ee^{3\pi\ii/4}$ to $w=\infty\ee^{-\ii\pi/4}$ passing respectively above and below the origin, and in the last line we used the definition \eqref{eq:F-intro} of $F(\xi)$.
On the other hand, for $2\le j,k\le N+1$,
\begin{multline}
    P_{jk}(\xi):=E_{j-1}\int_{W^+_{jk}}\ee^{-\ii (w-\xi)^2}\frac{\dd w}{w} + \int_{W^-}\ee^{-\ii(w-\xi)^2}\frac{\dd w}{w}+K_{jk}\ee^{-\ii\xi^2}\\
        =(1-E_{j-1})\int_{W^-}\ee^{-\ii (w-\xi)^2}\frac{\dd w}{w} + (2\pi\ii E_{j-1}\delta_{j\ge k} + K_{jk})\ee^{-\ii\xi^2}.
\label{eq:P-second}
\end{multline}
Since by definition the first column of $\mathbf{P}(\xi)$ is the vector $\boldsymbol{\eta}=(E_0-1,\dots,E_{N}-1)^\top$, defining $f(\xi):=E_0\ee^{\ii\pi/4}\sqrt{\pi}F(\xi)$ it follows that 
\begin{multline}
    \det(\mathbf{P}(\xi))\\=\begin{vmatrix}
        E_0-1 & f(\xi)+(2\pi\ii E_0+K_{12})\ee^{-\ii\xi^2} & \cdots  & f(\xi)+(2\pi\ii E_0 + K_{1,N+1})\ee^{-\ii\xi^2}\\
        \vdots &  & \ee^{-\ii\xi^2}\mathbf{T} &   \\
        E_{N}-1 & &&
        \end{vmatrix}.
\end{multline}
where $\mathbf{T}$ is the $N\times N$ matrix with elements
\begin{equation}
    T_{mn}:=K_{m+1,n+1}+2\pi\ii E_{m}\delta_{m\ge n},\quad m,n=1,\dots,N.
\label{eq:T-def}
\end{equation}
If there exists a non-exceptional index so that by convention $N$ is not exceptional, then $E_{N}=E_0$ and subtracting the last row from the first using Lemma~\ref{lem:firstrow-lastrow} gives
\begin{equation}
    \det(\mathbf{P}(\xi))=f(\xi)\begin{vmatrix}
        0 & 1 & \cdots  & 1\\
        -\mathbf{d} &  & \ee^{-\ii\xi^2}\mathbf{T} &   
        \end{vmatrix}=E_0\ee^{\ii\pi/4}\sqrt{\pi}F(\xi)\begin{vmatrix}
        0 & 1 & \cdots  & 1\\
        -\mathbf{d} &  & \ee^{-\ii\xi^2}\mathbf{T} &   
        \end{vmatrix},
\end{equation}
where $\mathbf{d}:=(1-E_1,\dots,1-E_{N})^\top$. 
On the other hand, if all indices are exceptional, then without subtracting the last row from the first we apply Lemma~\ref{lem:firstrow-lastrow} and use $E_0=1$ to obtain the same result.
Expanding along the first row gives
\begin{equation}
    \det(\mathbf{P}(\xi))=E_0\ee^{\ii\pi/4}\sqrt{\pi}F(\xi)\ee^{-\ii (N-1)\xi^2}\sum_{n=1}^N\det(\mathbf{T}\mathop{\longleftarrow}^n\mathbf{d}).
\end{equation}
Comparing the definition of $\mathbf{J}$ in \eqref{eq:XS-J-def} with \eqref{eq:Kjk} and \eqref{eq:T-def}, we see that each column of $\mathbf{T}$ differs from the corresponding column of $\mathbf{J}$ by a multiple of $\mathbf{d}$ which makes no contribution to $\det(\mathbf{T}\displaystyle\mathop{\longleftarrow}^n\mathbf{d})$ for any $n$.
We conclude that as $t\to+\infty$,
\begin{equation}
\begin{aligned}
    \det(\mathbf{A}(t,2\xi\sqrt{t}))&=-E_0\ee^{-\ii\pi/4}\sqrt{\pi}F(\xi)\ee^{-\ii N\xi^2}\sum_{n=1}^N\det(\mathbf{J}\mathop{\longleftarrow}^n\mathbf{d}) + O(t^{-1/2}\ln(t)^2)\\
    &=-E_0\ee^{-\ii\pi/4}\sqrt{\pi}F(\xi)\ee^{-\ii N\xi^2}\Delta + O(t^{-1/2}\ln(t)^2),\quad t\to+\infty,
\end{aligned}
\label{eq:detA-asymp}
\end{equation}
where $\Delta$ is defined in \eqref{eq:XS-DeltaN}.  This asymptotic is uniform for bounded $\xi$.

\subsection{Computing $\det(\mathbf{B}(t,2\xi\sqrt{t}))$}
To compute $\mathbf{B}(t,2\xi\sqrt{t})$, we use the same decomposition of columns $k=2,\dots,N+1$ as for $\mathbf{A}(t,2\xi\sqrt{t})$ (since these columns are identical), writing column $k$ as $c_k\boldsymbol{\eta} +\boldsymbol{\gamma}^{(k)}$ using \eqref{eq:ck-remaining-columns}--\eqref{eq:Gammajk-remaining-columns}.  For column $1$ of $\mathbf{B}(t,2\xi\sqrt{t})$ we use \eqref{eq:B-first-column-large-t} and the fact that $\theta_0=-\pi/4$ while $\theta_m=3\pi/4$ for $m=1,\dots,N$ to get, after computing a Gaussian integral, that the first column of $\mathbf{B}(t,2\xi\sqrt{t})$ is given by $\mathbf{b}_1(t,2\xi\sqrt{t})=c_1\boldsymbol{\eta} + \boldsymbol{\gamma}^{(1)}$ where
\begin{equation}
    c_1=2t^{1/2}\int_0^{\infty\ee^{3\pi\ii/4}}\ee^{-\ii (\zeta-\xi)^2}\,\dd\zeta,\quad \boldsymbol{\gamma}^{(1)}=2t^{1/2}\sqrt{\pi}\ee^{-\ii\pi/4}E_0\mathbf{e}_1 + O(\ln(t)),
\end{equation}
and $\mathbf{e}_1$ denotes the first unit vector in $\mathbb{C}^{N+1}$, $\mathbf{e}_1:=(1,0,\dots,0)^\top$.
By column multilinearity, we then get as before,
\begin{equation}
    \det(\mathbf{B}(t,2\xi\sqrt{t}))=\det(\boldsymbol{\Gamma}) + \sum_{k=1}^{N+1}c_k\det(\boldsymbol{\Gamma}\mathop{\longleftarrow}^k\boldsymbol{\eta}).
\end{equation}
Since the first column of $\boldsymbol{\Gamma}$ is now proportional to $t^{1/2}$, we get $\det(\boldsymbol{\Gamma})=O(t^{1/2})$ as $t\to+\infty$.  Likewise, $c_1\det(\boldsymbol{\Gamma}\displaystyle\mathop{\longleftarrow}^1\boldsymbol{\eta})=O(t^{1/2})$.  The remaining terms are larger: for $k=2,\dots,N+1$ we have
\begin{multline}
c_k\det(\boldsymbol{\Gamma}\mathop{\longleftarrow}^k\boldsymbol{\eta})\\
\begin{aligned}
 &= \left(\frac{\ee^{-\ii\xi^2}}{2}\ln(t) + O(1)\right)\left(2t^{1/2}\sqrt{\pi}\ee^{-\ii\pi/4}E_0\det((\mathbf{e}_1,\boldsymbol{\gamma}^{(2)},\dots,\boldsymbol{\gamma}^{(N+1)})\mathop{\longleftarrow}^k\boldsymbol{\eta}) + O(\ln(t))\right)\\
&=t^{1/2}\ln(t)\sqrt{\pi}\ee^{-\ii\pi/4}E_0\ee^{-\ii\xi^2} \lim_{t\to+\infty}\det((\mathbf{e}_1,\boldsymbol{\gamma}^{(2)},\dots,\boldsymbol{\gamma}^{(N+1)})\mathop{\longleftarrow}^k\boldsymbol{\eta}) + o(t^{1/2}\ln(t)).
\end{aligned}
\end{multline}
Taking the indicated limit, columns $2,\dots,N+1$ of $(\mathbf{e}_1,\boldsymbol{\gamma}^{(2)},\dots,\boldsymbol{\gamma}^{(N+1)})$ converge to the corresponding columns of $\mathbf{P}(\xi)$ simplified in \eqref{eq:P-first}--\eqref{eq:P-second}; since one of these columns is then to be replaced by $\boldsymbol{\eta}=(E_0-1,\dots,E_{N}-1)^\top$, all contributions to columns of $\mathbf{P}(\xi)$ proportional to $\boldsymbol{\eta}$ are cancelled. Expanding the determinant along the first column then yields 
\begin{equation}
\begin{aligned}
\lim_{t\to+\infty}\det((\mathbf{e}_1,\boldsymbol{\gamma}^{(2)},\dots,\boldsymbol{\gamma}^{(N+1)})\mathop{\longleftarrow}^k\boldsymbol{\eta}) &= -\det(\ee^{-\ii\xi^2}\mathbf{T}\mathop{\longleftarrow}^{k-1}\mathbf{d}) \\ &= -\ee^{-\ii (N-1)\xi^2}\det(\mathbf{T}\mathop{\longleftarrow}^{k-1}\mathbf{d}),\quad k=2,\dots,N+1.
\end{aligned}
\end{equation}
Once again, the matrix $\mathbf{T}$ can be replaced with $\mathbf{J}$ defined in \eqref{eq:XS-J-def} and we conclude that 
\begin{equation}
\begin{aligned}
    \frac{\det(\mathbf{B}(t,2\xi\sqrt{t}))}{t^{1/2}\ln(t)}&=-E_0\ee^{-\ii\pi/4}\sqrt{\pi}\ee^{-\ii N\xi^2}\sum_{n=1}^N\det(\mathbf{J}\mathop{\longleftarrow}^n\mathbf{d}) + o(1)\\
    &=-E_0\ee^{-\ii\pi/4}\sqrt{\pi}\ee^{-\ii N\xi^2}\Delta + o(1),\quad t\to +\infty
\end{aligned}
\label{eq:detB-asymp}
\end{equation}
holds in the $L^\infty_\mathrm{loc}$ sense with respect to $\xi$.

\subsection{Combining the results and completing the proof}
  Combining \eqref{eq:detA-asymp} and \eqref{eq:detB-asymp} noting that the product of factors $-E_0\ee^{-\ii\pi/4}\sqrt{\pi}\ee^{-\ii N\xi^2}$ is obviously nonzero, one sees that provided $\Delta\neq 0$, 
\begin{equation}
    \frac{\det(\mathbf{A}(t,2\xi\sqrt{t}))}{\det(\mathbf{B}(t,2\xi\sqrt{t}))} = \frac{F(\xi)}{t^{1/2}\ln(t)} + o\left(\frac{1}{t^{1/2}\ln(t)}\right),\quad t\to +\infty
\end{equation}
holds uniformly for bounded $\xi$.  Using this result in Theorem~\ref{thm:XS-exact} then completes the proof of Theorem~\ref{thm:XS-main} in the general case.

\appendix
\section{Proof of Theorem~\ref{thm:XS-generic}}
The calculation of the reflection coefficient $\beta(\lambda)$ for $\lambda>0$ for initial data $u_0$ of the form \eqref{eq:XS-u0-rational} with residues $c_1,\dots,c_N$ subject to \eqref{eq:XS-u0-in-L1} was first accomplished in \cite{MillerW16a}.  The procedure expresses $\beta(\lambda)$ as an improper real integral (taking $\epsilon=1$ in \cite[Equations (22)--(25)]{MillerW16a})
\begin{equation}
    \beta(\lambda)=\ii\ee^{-2\pi(c_1+\cdots+c_N)}\int_\mathbb{R}\ee^{-\ii \lambda z}\ee^{-\ii L(z)}\left[u_0(z)-\sum_{n=1}^N\frac{v_n(\lambda)}{z-p_n}\right]\,\dd z, \quad\lambda>0,
\end{equation}
where $L'(z)=u_0(z)$ and $v_1(\lambda),\dots,v_N(\lambda)$ are the components of the solution $\mathbf{v}(\lambda)$ of the linear system
\begin{equation}
    \mathbf{M}(\lambda)\mathbf{v}(\lambda)=\mathbf{r}(\lambda),\quad\lambda>0,
\end{equation}
in which $\mathbf{M}(\lambda)\in\mathbb{C}^{N\times N}$ and $\mathbf{r}(\lambda)\in\mathbb{C}^N$ (denoted respectively $\mathbf{A}^>(\lambda)$ and $\mathbf{b}^>(\lambda)$ in \cite{MillerW16a}) have components
\begin{equation}
    M_{mn}(\lambda):=\int_{C_m}\frac{\ee^{-\ii\lambda z}\ee^{-\ii L(z)}\dd z}{z-p_n},\quad r_m(\lambda):=-\lambda\int_{C_m}\ee^{-\ii\lambda z}\ee^{-\ii L(z)}\,\dd z.
\end{equation}
Here, the contours $\{C_m\}_{m=1}^N$ (denoted $\{C_m^>\}_{m=1}^N$ in \cite{MillerW16a}) have the meaning explained in the introduction, with the additional property that the unbounded part(s) of $C_m$ tend to $z=\infty$ in the lower half-plane to ensure the convergence of the integrals for $\lambda>0$.  To be precise, we fix a number $R>\max\{|p_1|,\dots,|p_N|\}$, and if $m$ is a non-exceptional index, we take $C_m$ to be the path consisting of the vertical line from $z=-\ii\infty$ to $z=-\ii R$, followed by a positively oriented Jordan loop in the domain $|z|\le R$ beginning and ending at $z=-\ii R$ with $z=p_1,\dots,p_m$ in its interior and all other poles $z=p_{m+1},\dots,p_N$ and $p_1^*,\dots,p_N^*$ in its exterior, followed by the vertical line from $z=-\ii R$ to $z=-\ii\infty$.  If $m$ is exceptional, we omit the final ray and take the part of $C_m$ with $|z|\le R$ to be a simple arc from $z=-\ii R$ to $z=p_m$ avoiding all other poles.    
See Figure~\ref{fig:Cm>}.
\begin{figure}[h]
\begin{center}
    \begin{tikzpicture}[>=Stealth, scale=0.75]

  \draw[gray, thin, ->] (-2.5,0) -- (2.5,0);
  \draw[gray, thin, ->] (0,-2.5) -- (0,2.5);
  \draw[gray, thin, dashed] (0,0) circle (1.75);

  \draw[
  red, thick, decoration={markings, mark=at position 0.75 with {\arrow[red,thick]{Stealth}}},
  postaction={decorate}
  ]
  ({1.75*cos(-90)+0.1},{1.75*sin(-90)}) to[out=90, in=-90, looseness=1] (-1.2,-0.6) to[out=90, in=-135, looseness=1] (0,0) to[out=45, in=0, looseness=1] (0,1.4) to[out=180, in=90, looseness=1] (-1.6,0) to[out=-90, in=135, looseness=1] (-1.132,-1.132) to[out=-45, in=90, looseness=1] (-0.1,-1.75); 

  \draw[
    red, thick,
    decoration={
      markings,
      mark=at position 0.7 with {\arrow[red, thick]{Stealth}}
    },
    postaction={decorate}
  ]
  ({2.5*cos(-90)+0.1}, {2.5*sin(-90)}) -- ({1.75*cos(-90)+0.1}, {1.75*sin(-90)});

  \draw[
    red, thick,
    decoration={
      markings,
      mark=at position 0.7 with {\arrow[red, thick]{Stealth}}
    },
    postaction={decorate}
  ]
  ({1.75*cos(-90)-0.1}, {1.75*sin(-90)}) -- ({2.5*cos(-90)-0.1}, {2.5*sin(-90)});
  
  \fill (0,0) circle (2.5pt);
  \node[below right] at (-0.1,0.1) {$0$};

  \fill (-1,0.6) circle (2.5pt);
  \node[right] at (-1,0.6) {$p_1$};

  \fill (0,0.6) circle (2.5pt);
  \node[above] at (0,0.6) {$p_m$};

  \fill (0.7,0.6) circle (2.5pt);
  \node[right] at (0.7,0.6) {$p_N$};

  \fill (-1,-0.6) circle (2.5pt);
  \node[right] at (-1,-0.6) {$p_1^*$};

  \fill (0,-0.6) circle (2.5pt);
  \node[below] at (0,-0.6) {$p_m^*$};

  \fill (0.7,-0.6) circle (2.5pt);
  \node[right] at (0.7,-0.6) {$p_N^*$};

 \fill ({1.75*cos(-90)},{1.75*sin(-90)}) circle (2.5pt);
 \node[below left] at ({1.75*cos(-90)},{1.75*sin(-90)}) {$-\ii R$};

    \node[below right] at ({1.75*cos(-45)-0.18},{1.75*sin(-45)+0.2}) {$|z|=R$};

  \draw[gray] node at (2,2) {$\boxed{z}$};

\def\shift{9}

  \draw[gray, thin, ->] ({-2.5+\shift},0) -- ({2.5+\shift},0);
  \draw[gray, thin, ->] ({0+\shift},-2.5) -- ({0+\shift},2.5);
  \draw[gray, thin, dashed] ({0+\shift},0) circle (1.75);

  \draw[
  red, thick, decoration={markings, mark=at position 0.75 with {\arrow[red,thick]{Stealth}}},
  postaction={decorate}
  ]
  ({1.75*cos(-90)+\shift},{1.75*sin(-90)}) to[out=90, in=-90, looseness=1] ({-1.2+\shift},-0.6) to[out=90, in=-135, looseness=1] ({0+\shift},0) to[out=45, in=-45, looseness=1] ({0+\shift},0.6); 

  \draw[
    red, thick,
    decoration={
      markings,
      mark=at position 0.7 with {\arrow[red, thick]{Stealth}}
    },
    postaction={decorate}
  ]
  ({2.5*cos(-90)+\shift}, {2.5*sin(-90)}) -- ({1.75*cos(-90)+\shift}, {1.75*sin(-90)});

  \fill ({0+\shift},0) circle (2.5pt);
  \node[below right] at ({-0.1+\shift},0.1) {$0$};

  \fill ({-1+\shift},0.6) circle (2.5pt);
  \node[right] at ({-1+\shift},0.6) {$p_1$};

  \fill ({0+\shift},0.6) circle (2.5pt);
  \node[above] at ({0+\shift},0.6) {$p_m$};

  \fill ({0.7+\shift},0.6) circle (2.5pt);
  \node[right] at ({0.7+\shift},0.6) {$p_N$};

  \fill ({-1+\shift},-0.6) circle (2.5pt);
  \node[right] at ({-1+\shift},-0.6) {$p_1^*$};

  \fill ({0+\shift},-0.6) circle (2.5pt);
  \node[below] at ({0+\shift},-0.6) {$p_m^*$};

  \fill ({0.7+\shift},-0.6) circle (2.5pt);
  \node[right] at ({0.7+\shift},-0.6) {$p_N^*$};

  \fill ({1.75*cos(-90)+\shift},{1.75*sin(-90)}) circle (2.5pt);
  \node[below left] at ({1.75*cos(-90)+\shift},{1.75*sin(-90)}) {$-\ii R$};

  \node[below right] at ({1.75*cos(-45)-0.18+\shift},{1.75*sin(-45)+0.2}) {$|z|=R$};

  \draw[gray] node at ({2+\shift},2) {$\boxed{z}$};

\end{tikzpicture}    
\end{center}
\caption{The contour $C_m$ for non-exceptional $m$ (left) and for exceptional $m$ (right).}
\label{fig:Cm>}
\end{figure}
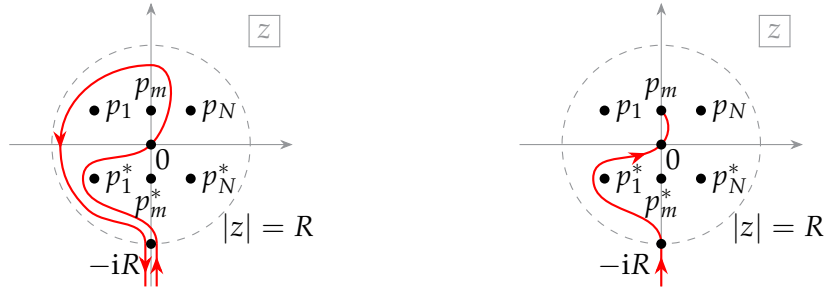

Integration by parts (using $\ee^{-\ii L(p_m)}=0$ if $m$ is an exceptional index) shows that $r_m(\lambda)$ can be equivalently written in the form
\begin{equation}
    r_m(\lambda)=\int_{C_m}\ee^{-\ii\lambda z}\ee^{-\ii L(z)}u_0(z)\,\dd z.
\end{equation}
Then, the definition of $\beta(\lambda)$ and the system for the coefficients $v_1(\lambda),\dots,v_N(\lambda)$ can be combined into a single linear system as follows:
\begin{multline}
    \begin{bmatrix} -\ii\ee^{2\pi (c_1+\cdots+c_N)} & \int_\mathbb{R}\frac{\ee^{-\ii\lambda z}\ee^{-\ii L(z)}\dd z}{z-p_1} & \cdots & \int_\mathbb{R}\frac{\ee^{-\ii\lambda z}\ee^{-\ii L(z)}\dd z}{z-p_N}\\
    0 & & & &\\
    \vdots & & \mathbf{M}(\lambda)&&\\
    0 & & & &
\end{bmatrix}\begin{bmatrix}\beta(\lambda)\\v_1(\lambda)\\\vdots\\v_N(\lambda)\end{bmatrix} \\ = \begin{bmatrix}\int_\mathbb{R}\ee^{-\ii\lambda z}\ee^{-\ii L(z)}u_0(z)\,\dd z\\
    \int_{C_1}\ee^{-\ii\lambda z}\ee^{-\ii L(z)}u_0(z)\,\dd z\\
    \vdots\\
    \int_{C_N}\ee^{-\ii\lambda z}\ee^{-\ii L(z)}u_0(z)\,\dd z
    \end{bmatrix}.
\end{multline}
Using Cramer's rule and expanding the denominator determinant along the first column yields a closed-form formula for $\beta(\lambda)$:
\begin{equation}
    \beta(\lambda)=\frac{\ii\ee^{-2\pi(c_1+\cdots+c_N)}\mathcal{N}(\lambda)}{\det(\mathbf{M}(\lambda))},
\label{eq:reflectioncoefficient}
\end{equation}
where $\mathcal{N}(\lambda)$ denotes the $(N+1)\times (N+1)$ determinant
\begin{equation}
    \mathcal{N}(\lambda):=\begin{vmatrix}
        \int_{\mathbb{R}}\ee^{-\ii \lambda z}\ee^{-\ii L(z)}u_0(z)\,\dd z & \int_{\mathbb{R}}\frac{\ee^{-\ii\lambda z}\ee^{-\ii L(z)}\dd z}{z-p_1} & \cdots & \int_{\mathbb{R}}\frac{\ee^{-\ii\lambda z}\ee^{-\ii L(z)}\dd z}{z-p_N}\\
        \int_{C_1}\ee^{-\ii\lambda z}\ee^{-\ii L(z)}u_0(z)\,\dd z & & & &\\
        \vdots & &\mathbf{M}(\lambda) & &\\
        \int_{C_N}\ee^{-\ii\lambda z}\ee^{-\ii L(z)}u_0(z)\,\dd z & & & &
    \end{vmatrix}.
\label{eq:RC-numerator-determinant}
\end{equation}
Since it is more convenient to work with absolutely convergent integrals, we 
consider the contour $\mathbb{R}$ in the first row of $\mathcal{N}(\lambda)$ to be deformed downwards at both ends toward $z=-\ii\infty$.

The first goal is to characterize the analytic properties of the matrix elements $M_{mn}(\lambda)$ in a right-neighborhood of $\lambda=0$.  If $m$ is non-exceptional, the ratio of values of $\ee^{-\ii L(z)}$ at corresponding points on the final and initial vertical segments is a constant, namely $E_{m}$ given by \eqref{eq:XS-En-non-exceptional}.  Therefore,
\begin{equation}
    M_{mn}(\lambda)=\int_{C_m\cap\{|z|\le R\}}\frac{\ee^{-\ii\lambda z}\ee^{-\ii L(z)}}{z-p_n}\,\dd z + (E_{m}-1)\int_{-\ii R}^{-\ii \infty}\frac{\ee^{-\ii \lambda z}\ee^{-\ii L(z)}}{z-p_n}\,\dd z,
\label{eq:Amn-magenta}
\end{equation}
where in the second term we take the value for $\ee^{-\ii L(z)}$ for the initial vertical segment, which tends to $1$ as $z\to -\ii\infty$.
The same holds if $m$ is exceptional and $E_m=0$ according to \eqref{eq:XS-En-exceptional}, since the final vertical segment of $C_m$ is absent and we observe that the first term is convergent at the finite endpoint $z=p_m$ regardless of whether or not $m=n$ because $\ee^{-\ii L(z)}$ vanishes to at least first order at $z=p_m$.

The first term on the right-hand side in \eqref{eq:Amn-magenta} is an entire function of $\lambda$, with power series expansion
\begin{equation}
    F_{mn}(\lambda):=\int_{C_m\cap\{|z|\le R\}}\frac{\ee^{-\ii\lambda z}\ee^{-\ii L(z)}}{z-p_n}\,\dd z = \sum_{k=0}^\infty\frac{(-\ii\lambda)^k}{k!}\int_{C_m\cap\{|z|\le R\}}\frac{z^k\ee^{-\ii L(z)}}{z-p_n}\,\dd z.
\label{eq:Fmn-entire}
\end{equation}
For the second term, since $|z|\ge R$, we may use the convergent Laurent series of $\ee^{-\ii L(z)}/(z-p_n)$ in the integrand:
\begin{equation}
    \frac{\ee^{-\ii L(z)}}{z-p_n}=\sum_{k=1}^\infty\frac{\mu_k^{(n)}}{z^k},\quad |z|\ge R,\quad \mu_1^{(n)}=1.
    \label{eq:LaurentSeries}
\end{equation}
In particular, this implies that the function $g_n(\lambda)$ defined by
\begin{equation}
g_n(\lambda):=\frac{1}{2\pi\ii}    \oint_{|z|=R_1>R}\frac{\ee^{-\ii\lambda z}\ee^{-\ii L(z)}}{z-p_n}\,\dd z = \sum_{k=1}^\infty \frac{\mu_k^{(n)}}{2\pi\ii}\oint_{|z|=R_1>R}\frac{\ee^{-\ii\lambda z}}{z^k}\,\dd z = \sum_{k=1}^\infty\frac{\mu_k^{(n)}}{(k-1)!}(-\ii\lambda)^{k-1}
\label{eq:Gn-entire}
\end{equation}
is an entire function with $g_n(0)=1$.  Using \eqref{eq:LaurentSeries} in the integrand,
\begin{equation}
    \int_{-\ii R}^{-\ii\infty}\frac{\ee^{-\ii\lambda z}\ee^{-\ii L(z)}}{z-p_n}\,\dd z = \sum_{k=1}^\infty\mu_k^{(n)}\int_{-\ii R}^{-\ii\infty}\frac{\ee^{-\ii\lambda z}\,\dd z}{z^k}=
    \sum_{k=1}^\infty\mu_k^{(n)}(\ii \lambda)^{k-1}\int_{R\lambda}^\infty\frac{\ee^{-w}\,\dd w}{w^k}.
\label{eq:outer-integral-with-pole-expansion}
\end{equation}
The integral over $w$ can be expanded as follows:
\begin{equation}
\begin{aligned}
    \int_{R\lambda}^\infty\frac{\ee^{-w}\,\dd w}{w^k} &= \int_1^\infty\frac{\ee^{-w}\,\dd w}{w^k} + \int_{R\lambda}^1\frac{\ee^{-w}\,\dd w}{w^k} \\ &= 
    \int_1^\infty\frac{\ee^{-w}\,\dd w}{w^k} +\sum_{j=0}^\infty\frac{(-1)^j}{j!}\int_{R\lambda}^1 w^{j-k}\,\dd w\\
    &=\int_1^\infty\frac{\ee^{-w}\,\dd w}{w^k} + \mathop{\sum_{j=0}^\infty}_{j\neq k-1}\frac{(-1)^j(1-(R\lambda)^{j-k+1})}{j!(j-k+1)} - \frac{(-1)^{k-1}}{(k-1)!}\ln(R\lambda)\\
    &=\frac{(-1)^k}{(k-1)!}\ln(\lambda) +c_k(R)-\frac{1}{(R\lambda)^{k-1}}\mathop{\sum_{j=0}^\infty}_{j\neq k-1}\frac{(-R\lambda)^j}{j!(j-k+1)},
\end{aligned}
\end{equation}
where
\begin{equation}
    c_k(R):=\int_1^\infty\frac{\ee^{-w}\,\dd w}{w^k} +\mathop{\sum_{j=0}^\infty}_{j\neq k-1}\frac{(-1)^j}{j!(j-k+1)} +\frac{(-1)^k\ln(R)}{(k-1)!}.
\end{equation}
We note in passing that 
\begin{equation}
c_1(R)=-\ln(R)-\gamma    
\label{eq:Euler}
\end{equation}
where $\gamma$ is the Euler constant.
Consequently,
\begin{multline}
    \int_{-\ii R}^{-\ii\infty}\frac{\ee^{-\ii\lambda z}\ee^{-\ii L(z)}}{z-p_n}\,\dd z = -\ln(\lambda)g_n(\lambda) \\+     
    \sum_{k=1}^\infty\mu_k^{(n)}\left[c_k(R)(\ii\lambda)^{k-1} - \left(\frac{\ii}{R}\right)^{k-1}\mathop{\sum_{j=0}^\infty}_{j\neq k-1}\frac{(-R\lambda)^j}{j!(j-k+1)}\right],
\end{multline}
where $g_n(\lambda)$ is the entire function defined by \eqref{eq:Gn-entire}.  The sum on the second line can be shown to be analytic for $|\lambda|<R^{-1}$, so combining with the entire function $F_{mn}(\lambda)$ defined by \eqref{eq:Fmn-entire} and using \eqref{eq:Amn-magenta} gives
\begin{equation}
    M_{mn}(\lambda)=H_{mn}(\lambda) -(E_{m}-1)g_n(\lambda)\ln(\lambda),
\label{eq:Amn-exact}
\end{equation}  
where $H_{mn}(\lambda)$ is analytic for $|\lambda|<1/R$.  

Since the second term in \eqref{eq:Amn-exact} is an element of a rank-$1$ matrix, in particular we have deduced the analytic structure of $\det(\mathbf{M}(\lambda))$ near $\lambda=0$:
\begin{equation}
    \det(\mathbf{M}(\lambda)) = \ln(\lambda)\sum_{n=1}^Ng_n(\lambda)\det(\mathbf{H}(\lambda)\mathop{\longleftarrow}^n\mathbf{d}) + \det(\mathbf{H}(\lambda)),
\label{eq:denominator-determinant}
\end{equation}
where $\mathbf{d}:=(1-E_1,\dots,1-E_{N})^\top$ as defined in \eqref{eq:XS-DeltaN} and $\displaystyle \mathbf{H}(\lambda)\mathop{\longleftarrow}^n\mathbf{d}$ is the matrix $\mathbf{H}(\lambda)$ with its $n$th column replaced by $\mathbf{d}$.

Now we turn our attention to the integrals in the first row and column of the numerator determinant $\mathcal{N}(\lambda)$ defined in \eqref{eq:RC-numerator-determinant}.   Recall that if there exists at least one non-exceptional index, we have chosen to order the poles so that $N$ is non-exceptional.  In this case, we can relate the integrals in the first row of $\mathcal{N}(\lambda)$, columns $2,\dots,N+1$, to those in the last row by 
\begin{equation}
\begin{aligned}
    \int_\mathbb{R}\frac{\ee^{-\ii\lambda z}\ee^{-\ii L(z)}}{z-p_n}\,\dd z -\int_{C_N}\frac{\ee^{-\ii\lambda z}\ee^{-\ii L(z)}}{z-p_n}\,\dd z &= -E_{N}\oint_{|z|=R_1>R}\frac{\ee^{-\ii\lambda z}\ee^{-\ii L(z)}}{z-p_n}\,\dd z \\ &= -2\pi\ii E_{N}g_n(\lambda)\\
    &=-2\pi\ii\ee^{2\pi(c_1+\cdots+c_N)}g_n(\lambda),\quad\text{for $N$ non-exceptional},
\end{aligned}
\label{eq:first-row-2-N+1}
\end{equation}
for $n=1,\dots,N$, where $g_n(\lambda)$ is the entire function defined by \eqref{eq:Gn-entire} and we used \eqref{eq:XS-En-non-exceptional} to evaluate $E_N$.   On the other hand, if all indices $m=1,\dots,N$ are exceptional, there is no need to subtract the last row, because the path can be closed in the lower half-plane and there are no singularities in the upper half-plane.  Therefore, for $n=1,\dots,N$,
\begin{equation}
\begin{aligned}
    \int_\mathbb{R}\frac{\ee^{-\ii\lambda z}\ee^{-\ii L(z)}}{z-p_n}\,\dd z &= -2\pi\ii g_n(\lambda)\\
    &=-2\pi\ii\ee^{2\pi(c_1+\cdots+c_N)}g_n(\lambda),\quad\text{for all indices exceptional},
\end{aligned}
\label{eq:first-row-2-N+1-all-exceptional}
\end{equation}
where the last line follows by the definition of an exceptional index.

Now we work on the integrals in the first column of the numerator determinant $\mathcal{N}(\lambda)$ in \eqref{eq:RC-numerator-determinant}, rows $2,\dots,N+1$.  Proceeding as above, for $m=1,\dots,N$ we have
\begin{equation}
    \int_{C_m}\ee^{-\ii\lambda z}\ee^{-\ii L(z)}u_0(z)\,\dd z =
    f_m(\lambda) +(E_{m}-1)\int_{-\ii R}^{-\ii\infty}\ee^{-\ii\lambda z}\ee^{-\ii L(z)}u_0(z)\,\dd z
\label{eq:N-first-column-start}
\end{equation}
where $f_m(\lambda)$ is the entire function
\begin{equation}
    f_m(\lambda):=\int_{C_m\cap \{|z|\le R\}}\ee^{-\ii\lambda z}\ee^{-\ii L(z)}u_0(z)\,\dd z = \sum_{k=0}^\infty\frac{(-\ii\lambda)^k}{k!}
    \int_{C_m\cap\{|z|\le R\}}z^k\ee^{-\ii L(z)}u_0(z)\,\dd z.
\end{equation}
Since $L'(z)=u_0(z)$, $f_m(0)=\ii (E_{m}-1)\ee^{-\ii L(-\ii R)}$.  This holds whether or not $m$ is an exceptional index; in the exceptional case that $C_m$ terminates at $z=p_m$, we recall that $E_{m}=0$ and observe that $\ee^{-\ii L(z)}$ vanishes to at least first order at $z=p_m$.  For the second term on the right-hand side in \eqref{eq:N-first-column-start}, we use the Laurent expansion of $\ee^{-\ii L(z)}u_0(z)$ convergent for $|z|\ge R$:
\begin{equation}
    \ee^{-\ii L(z)}u_0(z)=\sum_{k=2}^\infty \frac{\nu_k}{z^k},\quad |z|\ge R.
\label{eq:Laurent-2}
\end{equation}
Note that since the left-hand side is the derivative of $\ii \ee^{-\ii L(z)}$ and $L(z)$ is single valued in the domain $|z|\ge R$ with $L(z)\to 0$ as $z\to\infty$, 
\begin{equation}
    \ee^{-\ii L(z)}=1+\sum_{k=2}^\infty\frac{\ii\nu_k}{(k-1)z^{k-1}},\quad |z|\ge R.
\label{eq:ExpMinusIL-Laurent}
\end{equation}
Then $g(\lambda)$ defined by
\begin{equation}
\begin{aligned}
    g(\lambda):=\frac{1}{2\pi\ii}\oint_{|z|=R_1>R}\ee^{-\ii \lambda z}\ee^{-\ii L(z)}u_0(z)\,\dd z &= \sum_{k=2}^\infty\frac{\nu_k}{2\pi\ii}\oint_{|z|=R_1>R}\frac{\ee^{-\ii\lambda z}}{z^k}\,\dd z \\ &= \sum_{k=2}^\infty\frac{\nu_k}{(k-1)!}(-\ii\lambda)^{k-1}
\end{aligned}
\end{equation}
is entire, and $g(0)=0$.
Furthermore, integrating the Laurent series \eqref{eq:Laurent-2} term-by-term as was done in \eqref{eq:outer-integral-with-pole-expansion} we get
\begin{multline}
    \int_{-\ii R}^{-\ii\infty}\ee^{-\ii\lambda z}\ee^{-\ii L(z)}u_0(z)\,\dd z \\
\begin{aligned}
    &= \sum_{k=2}^\infty \nu_k(\ii\lambda)^{k-1}\int_{R\lambda}^\infty\frac{\ee^{-w}\,\dd w}{w^k} \\ &= -\ln(\lambda)g(\lambda) 
    +\sum_{k=2}^\infty\nu_k\left[c_k(R)(\ii\lambda)^{k-1}-\left(\frac{\ii}{R}\right)^{k-1}\mathop{\sum_{j=0}^\infty}_{j\neq k-1}\frac{(-R\lambda)^j}{j!(j-k+1)}\right],
\end{aligned}
\end{multline}
where again the sum on the last line is analytic for $|\lambda|<R^{-1}$, taking the value 
\begin{equation}
\sum_{k=2}^\infty\frac{\nu_k}{k-1}\left(\frac{\ii}{R}\right)^{k-1} = -\ii (\ee^{-\ii L(-\ii R)}-1)
\end{equation}
at $\lambda=0$, where we used \eqref{eq:ExpMinusIL-Laurent}.  Combining the results shows that
\begin{equation}
    \int_{C_m}\ee^{-\ii\lambda z}\ee^{-\ii L(z)}u_0(z)\,\dd z = h_m(z)-(E_{m}-1)\ln(\lambda)g(\lambda),\quad m=1,\dots,N,
\label{eq:first-column-exact}
\end{equation}
where $h_m(\lambda)$ is analytic for $|\lambda|<R^{-1}$ with 
\begin{equation}
h_m(0)=\ii (E_{m}-1)\ee^{-\ii L(-\ii R)} -\ii(E_{m}-1)(\ee^{-\ii L(-\ii R)}-1) = \ii (E_{m}-1).
\end{equation}
In other words, if $\mathbf{h}(\lambda):=(h_1(\lambda),\dots,h_N(\lambda))^\top$, then $\mathbf{h}(0)=-\ii\cdot\mathbf{d}$, where $\mathbf{d}:=(1-E_1,\dots,1-E_N)^\top$ as defined in the introduction.

It remains to study the integral in the upper left-hand corner of the numerator determinant.  If there exists a non-exceptional index, so by our convention $N$ is non-exceptional, then
\begin{multline}
    \int_\mathbb{R}\ee^{-\ii\lambda z}\ee^{-\ii L(z)}u_0(z)\,\dd z - \int_{C_N}\ee^{-\ii\lambda z}\ee^{-\ii L(z)}u_0(z)\,\dd z\\
\begin{aligned}
    &= -E_{N}\oint_{|z|=R_1>R}\ee^{-\ii\lambda z}\ee^{-\ii L(z)}u_0(z)\,\dd z \\ &=-2\pi\ii E_{N}g(\lambda)\\
    &=-2\pi\ii\ee^{2\pi(c_1+\cdots+c_N)}g(\lambda),\quad \text{for $N$ non-exceptional}.
\end{aligned}
\label{eq:first-row-1}
\end{multline}
Otherwise, without subtracting the lower-left-hand corner element, we close the contour in the lower half-plane and get
\begin{equation}
\begin{aligned}
    \int_\mathbb{R}\ee^{-\ii\lambda z}\ee^{-\ii L(z)}u_0(z)\,\dd z &= -2\pi\ii g(\lambda)\\
    &=-2\pi\ii\ee^{2\pi(c_1+\cdots+c_N)}g(\lambda),\quad\text{for all indices exceptional}.
\label{eq:first-row-1-all-exceptional}
\end{aligned}
\end{equation}

Now we use these results to analyze the numerator determinant $\mathcal{N}(\lambda)$ defined in \eqref{eq:RC-numerator-determinant}. 
If $N$ is non-exceptional, we subtract the last row from the first using \eqref{eq:first-row-2-N+1} and \eqref{eq:first-row-1}; otherwise all indices are exceptional and we simply use \eqref{eq:first-row-2-N+1-all-exceptional} and \eqref{eq:first-row-1-all-exceptional}.  Either way, the result is the same:
\begin{equation}
\mathcal{N}(\lambda)
    =-2\pi\ii \ee^{2\pi(c_1+\cdots+c_N)}\begin{vmatrix}
        g(\lambda) & g_1(\lambda) & \cdots & g_N(\lambda)\\
        \int_{C_1}\ee^{-\ii\lambda z}\ee^{-\ii L(z)}u_0(z)\,\dd z & & & &\\
        \vdots & &\mathbf{M}(\lambda) & &\\
        \int_{C_N}\ee^{-\ii\lambda z}\ee^{-\ii L(z)}u_0(z)\,\dd z & & & &
    \end{vmatrix}.
\end{equation}
Then, using \eqref{eq:Amn-exact} and \eqref{eq:first-column-exact} we express $\mathcal{N}(\lambda)$ in block form using the column vectors $\mathbf{h}(\lambda)$, $\mathbf{g}(\lambda):=(g_1(\lambda),\dots,g_N(\lambda))^\top$, $\mathbf{d}$, and the matrix $\mathbf{H}(\lambda)$:
\begin{equation}
    \mathcal{N}(\lambda)=-2\pi\ii\ee^{2\pi(c_1+\cdots+c_N)}\begin{vmatrix}
        g(\lambda) & \mathbf{g}(\lambda)^\top\\\mathbf{h}(\lambda)+\ln(\lambda)g(\lambda)\mathbf{d} & \mathbf{H}(\lambda)+\ln(\lambda)\mathbf{d}\mathbf{g}(\lambda)^\top
    \end{vmatrix}.
\end{equation}
Adding $(E_{m-1}-1)\ln(\lambda)=-d_{m-1}\ln(\lambda)$ times the first row to row $m$, $m=2,\dots,N+1$ removes all logarithms, yielding
\begin{equation}
\mathcal{N}(\lambda)  
    =-2\pi\ii \ee^{2\pi(c_1+\cdots+c_N)}\begin{vmatrix} g(\lambda) & \mathbf{g}(\lambda)^\top\\
    \mathbf{h}(\lambda) & \mathbf{H}(\lambda)
    \end{vmatrix}.
\end{equation}
Therefore, $\mathcal{N}(\lambda)$ is obviously an analytic function of $\lambda$ at $\lambda=0$, and expanding along the first row gives
\begin{equation}
\mathcal{N}(\lambda) 
    =-2\pi\ii \ee^{2\pi(c_1+\cdots+c_N)}\left(g(\lambda)\det(\mathbf{H}(\lambda))-\sum_{n=1}^Ng_n(\lambda)\det(\mathbf{H}(\lambda)\mathop{\longleftarrow}^n\mathbf{h}(\lambda))\right).   
\label{eq:numerator-determinant}
\end{equation}
Using \eqref{eq:denominator-determinant} and \eqref{eq:numerator-determinant} in \eqref{eq:reflectioncoefficient} gives
\begin{equation}
    \beta(\lambda)=2\pi\frac{\displaystyle g(\lambda)\det(\mathbf{H}(\lambda))-\sum_{n=1}^Ng_n(\lambda)\det(\mathbf{H}(\lambda)\mathop{\longleftarrow}^n\mathbf{h}(\lambda))}{\displaystyle \ln(\lambda)\sum_{n=1}^Ng_n(\lambda)\det(\mathbf{H}(\lambda)\mathop{\longleftarrow}^n\mathbf{d}) + \det(\mathbf{H}(\lambda))}.
\end{equation}
Using $\mathbf{h}(0)=-\ii\cdot\mathbf{d}$, this can be written in the claimed form \eqref{eq:XS-rational-beta-general} where $\varphi_1(\lambda)$, $\varphi_2(\lambda)$, and $\varphi_3(\lambda)$ given by
\begin{equation}
    \varphi_1(\lambda):=\sum_{n=1}^Ng_n(\lambda)\det(\mathbf{H}(\lambda)\mathop{\longleftarrow}^n\mathbf{d}),
\end{equation}
\begin{equation}
    \varphi_2(\lambda):=2\pi\frac{g(\lambda)}{\lambda}\det(\mathbf{H}(\lambda))-2\pi\sum_{n=1}^Ng_n(\lambda)\det\left(\mathbf{H}(\lambda)\mathop{\longleftarrow}^n\left[\frac{\mathbf{h}(\lambda)-\mathbf{h}(0)}{\lambda}\right]\right)
\end{equation}
and
\begin{equation}
    \varphi_3(\lambda):=\det(\mathbf{H}(\lambda)).
\end{equation}
are all analytic at $\lambda=0$, recalling that the entire function $g(\lambda)$ satisfies $g(0)=0$.

Next we show that $\varphi_1(0)=\Delta$.  Since each function $g_n(\lambda)$ satisfies $g_n(0)=1$, we have
\begin{equation}
    \varphi_1(0)=\sum_{n=1}^N\det(\mathbf{H}(0)\mathop{\longleftarrow}^n\mathbf{d}).
\end{equation}
We simplify $H_{mn}(0)$ as follows:
\begin{equation}
    H_{mn}(0)=F_{mn}(0)+(E_{m}-1)\sum_{k=1}^\infty\mu_k^{(n)}\left.\left[c_k(R)(\ii\lambda)^{k-1}-\left(\frac{\ii}{R}\right)^{k-1}\mathop{\sum_{j=0}^\infty}_{j\neq k-1}\frac{(-R\lambda)^j}{j!(j-k+1)}\right]\right|_{\lambda=0}.
\end{equation}
Using \eqref{eq:LaurentSeries} with $z=-\ii R$ and \eqref{eq:Euler}, we find
\begin{equation}
\begin{aligned}
    H_{mn}(0)
    &=\int_{C_m\cap\{|z|\le R\}}\frac{\ee^{-\ii L(z)}\,\dd z}{z-p_n} +(E_{m}-1)\left(\mu_1^{(n)}c_1(R)+\sum_{k=2}^\infty\frac{\mu_k^{(n)}}{k-1}\left(\frac{\ii}{R}\right)^{k-1}\right)\\
    &=\int_{C_m\cap\{|z|\le R\}}\frac{\ee^{-\ii L(z)}\,\dd z}{z-p_n} + (E_{m}-1)\left(\int^{-\ii\infty}_{-\ii R}\left[\frac{\ee^{-\ii L(z)}}{z-p_n}-\frac{1}{z}\right]\,\dd z-\ln(R)-\gamma\right)
\end{aligned}
\end{equation}
In the second integral on the right-hand side, $\ee^{-\ii L(z)}$ refers to the branch that is single-valued for $|z|>R$ and that tends to $1$ as $z\to\infty$, while in the first integral $\ee^{-\ii L(z)}$ refers to the analytic continuation of the same function into $|z|\le R$ along $C_m$.  Now
\begin{equation}
    \int_{-\ii R}^{-\ii\infty}\left[\frac{\ee^{-\ii L(z)}}{z-p_n}-\frac{1}{z}\right]\,\dd z = \int_{-\ii R}^{-\ii\infty}\frac{\ee^{-\ii L(z)}-1}{z-p_n}\,\dd z + \log\left(\frac{-\ii R}{-\ii R-p_n}\right).
\end{equation}
We recall from the introduction the splitting at $z=0$ of the contour $C_m$ into an an initial arc $C_{m}^{-}$ from $z=\infty$ to $z=0$ along which $\ee^{-\ii L(z)}\to 1$ as $z\to\infty$ and a terminal arc $C_{m}^{+}$.  If the index $m$ is non-exceptional, then  $C_{m}^{+}$ goes from $z=0$ to $z=\infty$ along which $\ee^{-\ii L(z)}\to E_{m}$ as $z\to\infty$; otherwise, $C_{m}^{+}$ terminates at $p_m$ and $E_m=0$.  In both cases the results can therefore be combined as
\begin{equation}
\begin{aligned}
    H_{mn}(0)&=\int_{C_{m}^{-}}\frac{\ee^{-\ii L(z)}-1}{z-p_n}\,\dd z + \int_{C_{m}^{+}}\frac{\ee^{-\ii L(z)}-E_{m}}{z-p_n}\,\dd z \\ &\quad\quad+ \int_{C_{m}^{-}\cap \{|z|\le R\}}\frac{\dd z}{z-p_n} + E_{m}\int_{C_{m}^{+}\cap \{|z|\le R\}}\frac{\dd z}{z-p_n} \\
    &\qquad\qquad+ (E_{m}-1)\left(\log\left(\frac{-\ii R}{-\ii R-p_n}\right)-\ln(R)-\gamma\right)\\
    &=\int_{C_{m}^{-}}\frac{\ee^{-\ii L(z)}-1}{z-p_n}\,\dd z + \int_{C_{m}^{+}}\frac{\ee^{-\ii L(z)}-E_{m}}{z-p_n}\,\dd z \\ &\quad\quad+2\pi\ii E_{m}\delta_{m\ge n}\\
    &\qquad\qquad+(E_{m}-1)\left(\log\left(\frac{-\ii R}{-\ii R-p_n}\right)-\ln(R)-\gamma-\int_{C_{m}^{-}\cap\{|z|\leq R\}}\frac{\dd z}{z-p_n}\right)\\
    &=\int_{C_{m}^{-}}\frac{\ee^{-\ii L(z)}-1}{z-p_n}\,\dd z + \int_{C_{m}^{+}}\frac{\ee^{-\ii L(z)}-E_{m}}{z-p_n}\,\dd z \\ &\quad\quad+2\pi\ii E_{m}\delta_{m\ge n} + (E_{m}-1)\left(-\frac{\ii\pi}{2}-\gamma-\log(-p_n)\right).
\end{aligned}
\end{equation}
Since the last term is a contribution to each column of $\mathbf{H}(0)$ that is a multiple of $\mathbf{d}$, it makes no contribution to $\displaystyle \det(\mathbf{H}(0)\mathop{\longleftarrow}^n\mathbf{d})$ for any $n$; therefore we may also write 
\begin{equation}    \varphi_1(0)=\sum_{n=1}^N\det(\mathbf{J}\mathop{\longleftarrow}^n\mathbf{d}),
\end{equation}
where $\mathbf{J}$ is the $N\times N$ matrix whose elements are defined by \eqref{eq:XS-J-def}.  Comparing with \eqref{eq:XS-DeltaN} then proves that $\varphi_1(0)=\Delta$ as desired.

Finally, we observe that according to the representation \eqref{eq:XS-rational-beta-general},
\begin{equation}
    \Delta=\varphi_1(0)\neq 0\implies \beta(\lambda)=\frac{2\pi\ii}{\ln(\lambda)} + O\left(\frac{1}{\ln(\lambda)^2}\right),\quad \lambda\downarrow 0.
\end{equation}
This proves the genericity of the initial datum $u_0$ and completes the proof of Theorem~\ref{thm:XS-generic}.

\bibliographystyle{siamplain}
\bibliography{references}

\end{document}